\documentclass[12pt]{article}
\usepackage{amsmath, graphicx, amsfonts,amssymb, calrsfs}
\usepackage{amsfonts,mathrsfs, color, amsthm}
\usepackage{bm}
\usepackage{dsfont}
\usepackage{subcaption}

\usepackage{enumitem}
\usepackage{hyperref}
\hypersetup{
colorlinks   = true, 
urlcolor     = blue, 
linkcolor    = blue, 
citecolor    = red 
}

\def\sphere{S^{n-1}}
\def\N{\mathbb{N}}
\def\Rn{{\mathbb R^n}}
\def\R{\mathbb{R}}

\def\A{\mathds{A}}
\def\B{\mathds{B}}
\def\E{\mathds{E}}
\def\Ln{\mathscr{L}_C}
\def\Oc{\Omega_{C^\circ}}
\def\On{\Omega_C}

\def\R{\mathbb{R}}
\def\K{\mathscr{K}(C, \omega)}

\newtheorem{cor}{Corollary}[section]
\newtheorem{problem}{Problem}[section]

\newtheorem{theorem}{Theorem}[section]
\newtheorem{lemma}{Lemma}[section]
\newtheorem{remark}{Remark}[section]
\newtheorem{proposition}{Proposition}[section]
\newtheorem{corollary}{Corollary}[section]

\newtheorem{definition}{Definition}[section]

\def\B{\mathbb{B}}

\def\bt{\begin{theorem}}
\def\et{\end{theorem}}
\def\be{\begin{equation}}
\def\ee{\end{equation}}
\def\bl{\begin{lemma}}
\def\el{\end{lemma}}
\def\br{\begin{remark}}
\def\er{\end{remark}}
\def\bc{\begin{corollary}}
\def\ec{\end{corollary}}
\def\bd{\begin{definition}}
\def\ed{\end{definition}}
\def\bp{\begin{proposition}}
\def\ep{\end{proposition}}

\begin{document}
\title{The $L_p$ Minkowski problem for $C$-close sets: existence and continuity
\footnote{Keywords: $C$-close set, $C$-coconvex set, the $L_p$  Minkowski type problem, Monge-Amp\`{e}re equation.}}
\author{Wen Ai, Deping Ye  and Baocheng Zhu}
\date{}
\maketitle
	
\begin{abstract} 
Let $C$ be a pointed closed convex cone in $\mathbb{R}^n$ with nonempty interior, and let $S^{n-1}$ denote the unit sphere in $\mathbb{R}^n$. The $L_p$ Minkowski problem for $C$-close sets is to determine, for a real number $p$ and a nonzero finite Borel measure $\mu$ defined on $\Omega_{C^\circ}=S^{n-1}\cap \mathrm{int} C^{\circ}$, whether there exists a $C$-close set $\mathds{A}$ such that $\mu$ is the $L_p$ surface area measure of $\mathds{A}$.  In this paper, we will solve the problem for $p\in (0,1)$ and for $\mu$ being a nonzero finite Borel measure on $\Omega_{C^\circ}$. Moreover, we establish the continuity of solutions to the $L_p$ Minkowski problem for $p\in [0, 1]$ in several settings.

\vskip 2mm 2010 Mathematics Subject Classification: 52A20, 52A39.
\end{abstract}

\section{Introduction and overview of the main results}
\setcounter{equation}{0}
Although the classical Minkowski problem for convex bodies (i.e., convex compact subsets in $\Rn$ with nonempty interiors) originated from the seminal work of Minkowski \cite{Min1897,Min1903} at the turn of 20th century, its $L_p$ extension took almost a century to develop. In  \cite{Lut93}, Lutwak discovered the elegant variational formula leading to the $L_p$ surface area measure of convex bodies and initiated the study of the $L_p$ Minkowski problem. Lutwak's variational formula used the $L_p$ addition of convex bodies,  a natural extension of the Minkowski addition, which was first introduced by Firey \cite{F62}. Specifically,  for $p\geq 1$,  $a, b>0$, and $K, L$ two convex bodies containing the origin in their interiors, the $L_p$ addition is the convex body  $a\cdot_pK+_pb\cdot_pL$ defined by its support function:   \begin{align}\label{p-addition} 
h(a\cdot_pK+_pb\cdot_pL, \cdot)^p= ah(K, \cdot)^p+ b h(L, \cdot)^p. 
\end{align}  
Hereafter, for a closed convex set $E$, its support function  $h(E, \cdot):\Rn\rightarrow \R$ is defined by 
\begin{align} \label{support-E}
h(E, x)=\sup_{y\in E} \langle x,y \rangle  
\end{align} 
for $x\in \Rn $, where $ \langle x,y \rangle$ denotes the inner product of $x, y\in \Rn.$ When $p=1$, the $L_p$ addition reduces to the Minkowski addition. 

Let $V_n(E)$ denote the $n$-dimensional volume of $E$. The variational formula by Lutwak in \cite{Lut93} reads: for $p\geq 1$, and $K, L$ two convex bodies containing the origin in their interiors,  
\begin{equation} \label{p-variation-1}
\lim_{\varepsilon\rightarrow 0^+} \frac{V_n(K+_p\varepsilon\cdot _p L)-V_n(K)}{\varepsilon}=\frac{1}{p}\int_{\sphere} h(L, u)^ph(K, u)^{1-p}\,dS_{n-1}(K, u),
\end{equation} 
where $\sphere$ denotes the unit sphere and $S_{n-1}(K, \cdot)$ is the surface area measure of $K$ defined on $\sphere$. Note that
\begin{equation*}
S_{n-1}(K, \eta)=\mathscr{H}^{n-1}(\pmb\nu_K^{-1}(\eta))\,\,\,\text{for each Borel set} \,\,\,\eta\subseteq S^{n-1},
\end{equation*} 
where $\mathscr{H}^{n-1}$ denotes the $(n-1)$-dimensional Hausdorff measure and $\pmb\nu_K: \partial K\rightarrow S^{n-1}$ is the Gauss map of $K$. The $L_p$ surface area measure of $K$ for $p\in \R$ is then defined by  $$\,d S_{n-1, p}(K, \cdot)=h(K, \cdot)^{1-p}\,dS_{n-1}(K, \cdot).$$ The introduction of the $L_p$ surface area measure naturally motivates the $L_p$ Minkowski problem: {\em For a real number $p\in \R$ and a given nonzero finite Borel measure $\mu$  on $\sphere$, is it possible to find a convex body $K$ such that $\mu=S_{n-1, p}(K, \cdot)$? Moreover, if such convex bodies exist, what can be said about the uniqueness and continuity?} The classical and $L_p$ Minkowski problems have received extensive attention and also demonstrated their significance in many other areas, see e.g., \cite{Ale38,Ale42,BBCY19,C06,CW06,FJ38,HLW16,HLX15,HLYZ05,JLW15,JLZ16,LW13,LO95,LYZ04,U03,Zhu152,Zhu15,Zhu17} among others. The special case for $p=0$ is usually called the logarithmic Minkowski problem, and is particularly important in many aspects: it aims to characterize the cone-volume measure, see e.g., \cite{BHZ16, BH16,BH17,BLYZ13,CFL22,CLZ19, HL14,Stan02,Stan03,Sta08,Zhu14}. 

It took much longer to extend the Minkowski type problems from convex bodies to unbounded convex hypersurfaces. Motivated by the pioneer works \cite{KT14} by Khovanski\u{\i} and Timorin, and \cite{MR} by Milman and Rotem, Schneider in \cite{Sch18}  creatively developed an analogous Brunn-Minkowski theory for $C$-close sets, where the Minkowski problem for $C$-close sets is of central interest. Later, Yang, Ye and Zhu in \cite{YYZ22} took another step toward an $L_p$ Brunn-Minkowski theory for $C$-close sets. Due to the important applications and close connections of $C$-close sets to differential geometry, partial differential equations, commutative algebra and singularity theory  
(see, e.g., \cite{KK14,KR21,R00,R01,VK85}), the ($L_p$) Brunn-Minkowski theory for $C$-close sets (or more generally, unbounded closed convex sets) should be in great demand. In particular,  the body of work on the Minkowski type problems for unbounded closed convex sets continue growing, see e.g., \cite{CW95,HL21,P80,Sch18,Sch21,Sch24,Sch242,Sch243,SZ24,U84,WXZZ24,YYZ22}.  

We now introduce some basic terminologies in \cite{Sch18, YYZ22}. Let $C$ be a pointed closed convex cone in $\Rn$ with nonempty interior. Thus, $$C=\{\lambda x: \lambda\geq 0\ \text{and}\ x \in C\}$$ and  $C \cap (-C)=\{o\}$ with $-C=\{-y: y\in C\}$ and $o$ the origin of $\Rn$. Let $\Omega_{C^\circ}=\sphere \cap \mathrm{int} C^\circ,$ where $\mathrm{int}(E)$ denotes the interior of $E\subset\Rn$ and 
\begin{equation}\label{cone-11-dual}
C^\circ=\{x \in \R^n: \langle x,y \rangle  \leq 0\,\,\,\text{for all}\,\,\, y\in C\}
\end{equation} 
is the dual cone of $C.$ An unbounded closed convex set $\A \subseteq C$ is called a $C$-close set if $V_n(C\,\backslash\A)$ is finite. For convenience, we always let $A=C\,\backslash\A$, and call $A$ a $C$-coconvex set if $\A$ is $C$-close. If $A$  is bounded, then $\A$ is called a $C$-full set. The natural algebraic operations for $C$-close sets should resemble the $L_p$ addition in \eqref{p-addition}. Indeed, for $\A$ a $C$-close set, its support function $h_C(\A, \cdot): \Omega_{C^\circ}\rightarrow (-\infty, 0)$ defined by \eqref{support-E} is negative, that is, for $u\in \Omega_{C^\circ}$,  
\begin{align*} 
h_C(\A, u)=\sup_{y\in \A} \langle u, y \rangle .  
\end{align*}   
In view of this, one defines the support function of $A$ by $\overline{h}_C(A, \cdot)=-h_C(\A, \cdot)$. Then, the $p$-co-sum of $C$-coconvex sets for $p=1$ \cite{Sch18}  and for $0<p<1$  \cite{YYZ22} can be defined as follows:  
for $a, b>0$ and two $C$-close sets $\A_1$ and $\A_2$
with $A_1=C\,\backslash\A_1$ and $A_2=C\,\backslash\A_2$, define the $p$-co-sum of $A_1$ and $A_2$ by its support function 
\begin{equation*}
\overline{h}_C(a\cdot_pA_1 \oplus_p b\cdot_p A_2, u)^p= a\overline{h}_C(A_1, u)^p+ b\overline{h}_C(A_2, u)^p\ \ \ \text{for}\ \  u \in \Oc.
\end{equation*}  That is, for $p\in (0, 1]$, 
\begin{equation*}	
C\,\backslash (a\cdot_pA_1 \oplus_p b\cdot_p A_2)=C \cap \bigcap_{u\in \Oc} \big\{ x \in \mathbb{R}^n: x\cdot u\leq -\overline{h}_C(a\cdot_pA_1 \oplus_p b\cdot_p A_2, u)\big\}. 
\end{equation*} 
A variational formula similar to \eqref{p-variation-1}  (see \eqref{p-variation-2} in Section \ref{section-2} for details) leads to the $L_p$ surface area measure \cite{Sch18, YYZ22}. Namely, the $L_p$ surface area measure of a $C$-close set $\A$ for $p\in \R$ is defined by 
\begin{equation}\label{Sp}
dS_{n-1, p}(\A, \cdot)=(-h_C(\A, \cdot))^{1-p}dS_{n-1}(\A, \cdot), 
\end{equation} where ${S}_{n-1}(\A, \cdot)$ is the surface area measure of $\A$ defined by Schneider in \cite{Sch18}:  
\begin{equation*}
{S}_{n-1}(\A, \eta)=\mathcal{H}^{n-1} (
\pmb\nu^{-1}_{\A}(\eta)) \,\,\, \text{for any Borel set}\,\,\, \eta\subseteq \Oc,
\end{equation*} with  $\pmb\nu^{-1}_{\A}$ being the inverse Gauss map of $\A$. 
The following $L_p$ Minkowski problem for $C$-close sets (for $p\in \{0, 1\}$ by Schneider in \cite{Sch18} and for other $p\in \R$ by Yang, Ye and Zhu in \cite{YYZ22}) can be asked:  

\begin{problem}[{\bf The $L_p$ Minkowski problem for $C$-close sets}] \label{Lp-Minkowski-unbounded}  Let $\mu$ be a nonzero Borel measure defined on $\Omega_{C^\circ}$ and $p\in \R$. Is it possible to find a $C$-close set $\A$ such that $\mu=S_{n-1, p}(\A, \cdot)$? Moreover, if such $C$-close sets exist,  what can be said about the uniqueness and continuity?
\end{problem} Problem \ref{Lp-Minkowski-unbounded} is closely related to the following Monge-Amp\`{e}re equation:
\begin{equation*}
\det\big(\nabla^2 h(u)+h(u) I\big)=f(u)[-h(u)]^{p-1},
\end{equation*} 
for $u \in \Omega_{C^\circ}$, where  $h : \Omega_{C^\circ} \rightarrow (-\infty, 0)$ is the unknown convex function, $\nabla^2$ is the Hessian operator with respect to an orthonormal frame on $S^{n-1}$, $\det A$ denotes the determinant of $A$, and $I$ is the identity matrix.

When $p=1$, Problem \ref{Lp-Minkowski-unbounded} was solved by Schneider in \cite{Sch18} when $\mu$ is a nonzero finite Borel measure with its support concentrated on a compact subset of $\Oc$, and in \cite{Sch21} when $\mu$ is only assumed to be nonzero and finite.  Schneider also proved the uniqueness of solutions to Problem \ref{Lp-Minkowski-unbounded} for $C$-close sets in \cite{Sch18}, and the stability  of solutions to Problem \ref{Lp-Minkowski-unbounded} for $C$-full sets in \cite{Sch21}.  A solution to Problem \ref{Lp-Minkowski-unbounded} when $\mu$ is an infinite measure was provided in \cite{Sch24} by Schneider.  
When $p=0$, the $L_p$ surface area measure is just the cone-volume measure of $\A$, and Problem \ref{Lp-Minkowski-unbounded} is often called the logarithmic Minkowski problem. Schneider in \cite{Sch18} showed that every nonzero finite Borel measure on $\Oc$ is the cone-volume measure of a $C$-close set. He raised an open problem regarding the uniqueness of its solutions in \cite{Sch18}, and this has been completely solved by Yang, Ye and Zhu in \cite{YYZ22}. When $\mu$ is a nonzero finite Borel measure with a compact support, Yang, Ye and Zhu showed that Problem \ref{Lp-Minkowski-unbounded} can be solved for any $p \in \R$. Moreover, Yang, Ye and Zhu in \cite[Remark 5.1]{YYZ22} also showed that $\A_1=\A_2$ if $\A_1$ and $\A_2$ are two $C$-close sets and $p\in (0,1)$ such that $S_{n-1, p}(\A_1, \cdot)=S_{n-1, p}(\A_2, \cdot)$. These show  the uniqueness of Problem \ref{Lp-Minkowski-unbounded} for $p\in [0, 1]$, if it is solvable.  
 
Theorem \ref{ex&uniC-close} provides a solution to Problem \ref{Lp-Minkowski-unbounded} for $0\leq p\leq 1$ when $\mu$ is assumed to be a nonzero finite Borel measure on $\Oc.$  

\vskip 2mm \noindent {\bf Theorem \ref{ex&uniC-close}.} {\em Let $\mu$ be a nonzero finite Borel measure defined on $\Omega_{C^\circ}$. For $0 \leq p \leq 1$, there exists a unique $C$-close set $\A$ such that $\mu=S_{n-1, p}(\A, \cdot)$.} 

The proof of Theorem \ref{ex&uniC-close} heavily relies on the approximation techniques applied in \cite{Sch18, Sch21}. Such a method, together with the uniqueness of solutions to the $L_p$ Minkowski problem for $C$-close sets from Theorem \ref{ex&uniC-close}, motivates the problem regarding the continuity of the solutions to the $L_p$ Minkowski problem for $0\leq p\leq 1$. Let $\mathbb{N}$ be the set of all positive integers, $\mathbb{N}_0=\mathbb{N}\cup \{0\},$ and $\{\nu_i\}_{i \in \N_0}$ be a sequence of nonzero finite Borel measures on $\Oc$. By Theorem \ref{ex&uniC-close}, for $0\leq p\leq 1$, there exists a unique $C$-close set $\B_i$ such that $\nu_i=S_{n-1, p}(\B_i, \cdot)$ for each $i\in \mathbb{N}_0$. The continuity problem asks: {\em under what conditions on the sequence of measures $\{\nu_i\}_{i \in \N_0}$, does $\B_i$ converge to $\B_0$ (see Definition \ref{con-def})?} It will be proved in Lemma \ref{vague-con-Sp} that, if  $C$-compatible sets $\A_i$ converge to a $C$-compatible set $\A_0$, then $S_{n-1, p}(\A_i, \cdot)$ converges vaguely to $S_{n-1, p}(\A_0, \cdot)$ on $\Oc$ as $i \to \infty$. That is, for each continuous function $f: \Oc \to \R$ with compact support in $\Oc$, one has,
\begin{equation*}
\int_{\Oc} f(u)dS_{n-1, p}(\A_i, \cdot) \to \int_{\Oc} f(u)dS_{n-1, p}(\A_0, \cdot).
\end{equation*}
It will be written as $S_{n-1, p}(\A_i, \cdot) \to S_{n-1, p}(\A_0, \cdot)$ vaguely on $\Oc$.
Based on this observation, we can prove the following result.  

\vskip 2mm \noindent {\bf Theorem \ref{con-mu}.}  {\em Let $0\leq p\leq 1$. Assume that $\nu_i$, $i\in \N_0$, are nonzero finite Borel measures such that  $\nu_i$ converges to $\nu_0$ vaguely on $\Oc$ and  $\sup \{\nu_i(\Oc): i \in \N_0\}<\infty.$   If,  for each $i\in \mathbb{N}_0$,  $\B_i$ is the unique $C$-close set solving Problem \ref{Lp-Minkowski-unbounded} (i.e., $\nu_i=S_{n-1, p}(\B_i, \cdot)$), then $\B_i \rightarrow \B_0$ as $i\rightarrow \infty$.}

Besides Theorem  \ref{con-mu},  we also obtain several continuity results regarding Problem \ref{Lp-Minkowski-unbounded} in other settings. For example, in Theorem \ref{con-p}, we prove the continuity of solutions to Problem \ref{Lp-Minkowski-unbounded} when $p_i\in [0, 1]$ for $i\in \mathbb{N}_0$ such that $p_i\rightarrow p$. 

\vskip 2mm \noindent {\bf Theorem \ref{con-p}.} {\em Let $0\leq p_i\leq 1$ for $i \in \mathbb{N}_0$, and let $\nu$ be a nonzero finite Borel measure  defined on $\Omega_{C^\circ}$. Suppose that $\E_i$ for $i\in \mathbb{N}_0$ solves Problem \ref{Lp-Minkowski-unbounded} for $p_i$ (i.e., $\nu=S_{n-1, p_i}(\E_i, \cdot)$). Then $\E_i \rightarrow \E_0$  as $p_i\to p_0$. }
	 
\section{Background and Preliminaries}\label{section-2}
\setcounter{equation}{0}
Let $n\geq 2$ be a positive integer, and $|x|$ denote the Euclidean norm of $x\in \Rn$. By $\langle x,y \rangle$ we mean the standard inner product of $x, y\in \Rn$. 
Denote by $S^{n-1}=\left\{ x \in \mathbb{R}^n : |x|=1\right\}$ the unit sphere and $B_n=\{x\in \Rn: |x|\leq 1\}$ the unit Euclidean ball. For $E\subseteq \Rn$, its closure, interior, boundary, and  volume are denoted by  $\overline{E}$, $\text{int} E$, $\partial E$, and $V_n(E)$ respectively. The symbol $o$ stands for the origin in $\Rn$. A subset $K \subset \mathbb{R}^n$ is convex if $\lambda x+(1-\lambda)y \in K$ for all $x, y \in K$ and $0 \leq \lambda\leq 1$. For $u\in \sphere$ and $a\in \R$, let $H(u, a)=\{x\in \Rn:  \langle x, u \rangle =a\}.$ The set $H(u, a)$ is a hyperplane with normal vector $u$. Associated to $H(u, a)$ are halfspaces $H^{-}(u, a)$ and $H^+(u, a)$ given by  $H^{-}(u, a)=\{x\in \Rn: \langle x, u \rangle \leq a\}$ and respectively, $H^{+}(u, a)=\{x\in \Rn: \langle x,u \rangle \geq a\}.$   
 
Let $\mathbf{C}(X)$ define the family of continuous functions defined on $X$. By  $\mathbf{C}_b(X)$ we mean all bounded continuous functions defined on $X$.  Let $\{\nu_i\}_{i \in \N_0}$ be a sequence of finite Borel measures on $X$. We say that $\nu_i$ converges to $\nu_0$ weakly on $X$ if 
\begin{equation} \label{def-weakly}
\int_X f(x)d\nu_i(x) \to \int_X f(x)d\nu_0(x)\ \ \text{for all}\ \  f\in \mathbf{C}_b(X).
\end{equation} 
For simplicity, we can write it as $\nu_i \to \nu_0$ weakly on $X$.
  
Let $C$ be a pointed closed convex cone with nonempty interior. Define $\Omega _C =\sphere \cap \mathrm{int} C$ and $\Omega _{C^{\circ}}=\sphere \cap \mathrm{int} C^{\circ}$ with $C^{\circ}$ the dual cone of $C$ in \eqref{cone-11-dual}. A fixed vector $\xi\in \sphere \setminus C^{\circ}$ can be found so that  $\langle x, \xi \rangle >0$ for all $x \in C\, \backslash\, \{o\}$. For $M \subseteq C$, let   $M_t=M \cap H^{-}_{t}$ with $H^{-}_{t}=H^{-}(\xi, t).$ In particular, let $C_t=C \cap H^{-}_{t}$ and $C_t$ is bounded for every $t>0$. 

An unbounded closed convex set $o\notin E \subset C$ is a  $C$-compatible set \cite{LYZ23}  (also known as a $C$-pseudo cone \cite{Sch24}) if 
\begin{equation*}
E=C \cap \bigcap_{u \in \Oc} \{ H_u^-: E \subseteq H_u^- \},
\end{equation*}
where $H_u^{-}=H^{-}(u, a)$ for some $a\in \R$. Clearly, if $\A$ is a $C$-close set, then $\A$ must be a $C$-compatible set, and in particular,   
\begin{equation}\label{Ccom-c-close}
\A=C \cap \bigcap_{u \in \Oc}   H^-(u, h_C(\A, u)).
\end{equation} When $\Oc$ in \eqref{Ccom-c-close} is replaced by a compact set $\omega\subset \Oc$, then $\A$ is called a $C$-determined set by $\omega$ formulated by: 
\begin{equation*}
\A=C \cap \bigcap_{u \in \omega}H ^{-} \left( u, h_C(\A, u)\right). 
\end{equation*} 
Denote by $\mathscr{K}(C, \omega)$ the set of all $C$-determined sets by $\omega$. For convenience, if $\A\in \mathscr{K}(C, \omega),$ we will write $\A=[C, \omega, -h_C(\A, \cdot)]$. In general, if $f:\omega\rightarrow (0, \infty)$ is a positive continuous function on $\omega$, the Wulff shape of $f$, denoted by $[C, \omega, f]$, is defined by  
\begin{equation}\label{Wullf-shape}
[C, \omega, f]=C\cap \bigcap_{u\in \omega}H^-(u, -f(u)). 
\end{equation}  
Clearly, a $C$-determined set must be $C$-full.  

Let us now state some established facts regarding the $C$-determined sets, which are essential for later context. The first of such results is \cite[Lemma 7]{Sch18}, which is stated as follows.   

\begin{lemma}\label{L1} Let $\omega\subset \Oc$ be a compact set. There is a constant $\widetilde{t}_0>0$ with the following property: if $Q \in \mathscr{K}(C, \omega)$ and $V_n(C\, \backslash Q)=1$, then $C\, \cap\,H_{\widetilde{t}_0}\subset Q $.
\end{lemma}

For a compact convex set $E$, its support function $h(E, \cdot)$ is given by \eqref{support-E}. We say that compact convex sets $E_i$ converge to a compact convex set $E$ if $h(E_i,\cdot)$ converges  to $h(E, \cdot)$ uniformly on $\sphere.$ We have the following 
Blaschke selection theorem (see e.g., \cite[Theorem 1.8.7]{Sch14}). 
\begin{theorem} \label{Blas-sel-com}   
Every uniformly bounded sequence of compact convex sets in $\Rn$ has a subsequence that converges to a compact convex set in $\Rn$. 
\end{theorem}

There is an analogous Blaschke selection theorem for $C$-compatible sets. To this end, let us first introduce the convergence of $C$-compatible sets defined in \cite{LYZ23}. For $t>0$ and $\A$ a $C$-compatible set, if $\A\cap C_t\neq \emptyset,$ then $\A\cap C_t$ must be a compact convex set. 
\begin{definition}\label{con-def}
Let $\{\A_i\}_{i \in \N_0}$ be a sequence of $C$-compatible sets. If there exists $t_0>0$ such that $\A_i \cap C_{t_0}\neq \emptyset$ for all $i \in \N$ and for all $t>t_0$,
\begin{equation*}
\A_i\cap C_t \to \A_0 \cap C_t \,\,\, \text{as}\,\,\, i\to \infty
\end{equation*}
in the Hausdorff metric, then $\A_i$ converges to $\A_0$ as $i \to \infty$. This convergence will be written by $\A_i \to \A_0$ as $i \to \infty$.
\end{definition}

Let $b(\A)$ be the distance from origin to the $C$-compatible set $\A$ \cite{Sch24}: 
\begin{equation*}
b(\A)=\min \{r>0: rB_n \cap \A\neq \emptyset\}.
\end{equation*} 
It can be easily checked that 
\begin{equation}\label{hb}
-h_C(\A, \cdot)\leq b(\A) \,\,\,\text{on}\,\,\,  \Oc\,\,\,\text{and}\,\,\,\rho_C(\A,\cdot)\geq b(\A)\,\,\,\text{on}\,\,\, \Omega_C,
\end{equation} where $\rho_C(\A, v)=\sup\{r>0: rv\in C\,\backslash \A\}$ for $v \in \On$. 

Suppose that $\{\A_j\}_{j=1}^{\infty}$ is a sequence of $C$-compatible sets, such that 
\begin{align}\label{two-bounds-1}
0<\beta_0 \leq b(\A_j) \leq  \beta_1 <\infty\ \ \mathrm{for} \ \ j\in \mathbb{N}  
\end{align}  
with $\beta_0, \beta_1$ two constants. Let  $\{t_k\}_{k \in \mathbb{N}}$ be an increasing sequence diverging to $\infty$  with $t_1 > \beta_1$. Then  for $j, k \in \mathbb{N}$, one has  
\begin{equation*} 
\emptyset\neq \A_{j} \cap C_{t_k} \subset C_{t_k}.
\end{equation*} 
Thus, for each $k\in \mathbb{N}$, the sequence $\{\A_{j} \cap C_{t_k}\}_{j\in \mathbb{N}}$ is a bounded sequence of compact convex sets. Applying the Blaschke selection theorem to $\{\A_{j} \cap C_{t_k}\}_{j\in \mathbb{N}}$, a subsequence, which will still be denoted by $\{\A_{j} \cap C_{t_k}\}_{j\in \mathbb{N}}$ for convenience, can be found so that  $\A_{j} \cap C_{t_k} \rightarrow M_k$ as $j \rightarrow \infty$ with $M_k\subseteq C_{t_k}$ a convex body in $\mathbb{R}^n$. In fact, one can further have $M_k\subset M_{k+1}$ for all $k\in \mathbb{N}$. For more details on how to construct $M_k$, please see \cite{LYZ23,Sch18,Sch21}. Let $\A=\cup_{k\in \mathbb{N}}M_k$, which is clearly an unbounded closed convex set in $C$. Moreover, as $b(\A_j)\geq \beta_0$ for all $j\in \mathbb{N}$, one can get $b(\A)>0$ and hence $o\notin \A$. Thus,  $\A$ is a $C$-compatible set. Due to $M_k\subset M_{k+1}$ for all $k\in \mathbb{N}$,   $\A \cap C_{t_k}=M_k$ for each $k \in \mathbb{N}$. The above argument can be summarized to the following Blaschke selection theorem for $C$-compatible sets, which was given in \cite[Lemma 1]{Sch24}. 

\begin{theorem}[\bf The Blaschke selection theorem for $C$-compatible sets]\label{Blas-unb-1} 
Every sequence of $C$-compatible sets whose distances  from the origin are bounded and bounded away from $0$ has a subsequence that converges to a $C$-compatible set. More precisely, if $\{\A_j\}_{j=1}^{\infty}$ is a sequence of $C$-compatible sets satisfying \eqref{two-bounds-1}, then there exists a convergent subsequence $\{\A_{j_k}\}_{k\in \mathbb{N}}$ such that $\A_{j_k}\rightarrow \A$ with $\A$ a $C$-compatible set.  
\end{theorem}
 
We now summarize the convergence of the $C$-determined sets \cite[Lemma 5 and Lemma 6]{Sch18} as follows. By $\mathbf{C}^+(\omega)$ we mean the set of all positive continuous functions on $\omega$. 

\begin{lemma}\label{L56}
Let $\omega\subset \Oc$ be a compact set. The following statements hold true.
\vskip 1mm \noindent i) If $f_i\in \mathbf{C}^+(\omega)$ for all $i\in \mathbb{N}_0$ such that  $f_i\rightarrow f_0$ uniformly on $\omega$ as $i\rightarrow \infty$, then  $$[C, \omega, f_i]\rightarrow [C, \omega, f_0].$$
	
\vskip 2mm \noindent ii) If $\{\A_j\}_{j\in \mathbb{N}} \subset \mathcal{K}(C, \omega)$ such that $\A_j\rightarrow \A_0$ for some $C$-full set $\A_0$, then $\A_0 \in \mathcal{K}(C, \omega)$.
\end{lemma} It has been proved (see, e.g., in the proof of \cite[Lemma 5.2]{LYZ23}) that $h_C(\A_i, \cdot)\to h_C(\A_0, \cdot)$ uniformly on $\omega$ if $\A_i \to \A_0$ for $\{\A_i\}_{i \in \N_0} \subset \K$ as $i \to \infty$. Indeed, it follows from Definition \ref{con-def} that $\A_i \to \A_0$ implies the existence of $t_0>0$, such that, for all $t>t_0$, $\A_i \cap C_t\to \A_0\cap C_t$ as $i \to \infty$ in the Hausdorff metric. We can take sufficiently large $t>t_0$ such that $\omega \subset \pmb\nu_{\A_i}(\A_i \cap C_t)$ for every $i \in \N_0$. Then \cite[Lemma 1.8.14]{Sch14} yields that  
\begin{equation} \label{unif}
h_C(\A_i, \cdot) \to h_C(\A_0, \cdot)\,\,\, \text{uniformly on}\,\,\,\omega.
\end{equation}

Let $\A$ be a $C$-compatible set. Then, $\A$ and its companion $A=C\setminus \A$ share the common effective boundary in $\partial \A \cap \partial A \subset \mathrm{int}C$.  Thus, it is natural to let  $$\overline{S}_{n-1, p}(A, \cdot)=S_{n-1, p}(\A, \cdot).$$ Let $\pmb\nu_{\A}: \partial \A\cap \text{int}C \rightarrow S^{n-1}$ be the Gauss map of $\A$. That is, for $\sigma \subseteq \partial \A\cap \text{int}C$,
\begin{equation*}
\pmb\nu_{\A}(\sigma)=\{ u \in \Oc: x \in H  \left( u, h_C(\A, u)\right)\ \ \text{for some}\ \  x \in \sigma \}. 
\end{equation*} Denote by $\pmb\nu^{-1}_{\A}$ the inverse Gauss map of $\A$. Schneider in \cite{Sch18} proved that   \begin{equation}\label{Vn}
V_n(C\,\backslash\A)=\frac{1}{n} \int_{\Oc} (-h_C(\A, u))dS_{n-1}(\A, u)=\frac{1}{n} \int_{\Oc} \overline{h}_C(A, u)d\overline{S}_{n-1}(A, u)=V_n(A). 
\end{equation} 
Let $\A_1, \A_2 \in \mathscr{K}(C, \omega)$, and $A_1=C\,\backslash\A_1$ and $A_2=C\,\backslash\A_2$. Similar to \eqref{p-variation-1}, one has the following variational formula \cite{YYZ22}: for $p\neq 0$,  
\begin{equation}\label{p-variation-2}
\lim_{\varepsilon\rightarrow 0^+} \frac{V_n(A_1 \oplus_p \varepsilon\cdot_p A_2)-V_n(A_1)}{\varepsilon}=\frac{1}{p}\int_{\omega} \overline{h}_C(A_2, u)^p  \,d\overline{S}_{n-1, p}(A_1, u).
\end{equation} 
In particular,   
\begin{align}\label{volume-Lp}
V_n(C\,\backslash \A)  = \frac{1}{n}\int_{\Oc} (-h_C(\A, u))^p dS_{n-1, p}(\A, u), 
\end{align} which can be obtained either by \eqref{p-variation-2} directly with $\A_1=\A_2=\A$, or by \eqref{Sp} and \eqref{Vn}.   

Yang, Ye and Zhu \cite[Lemma 4.3]{YYZ22} proved the weak convergence of the surface area measure of $C$-determined sets.  
\begin{lemma}\label{L2}
Let $\left\{ \A_i\right\}_{i \in \mathbb{N}_0} \subset \mathscr{K}(C, \omega)$. If $\A_i \rightarrow \A_0$, then $S_{n-1} \left( \A_i, \cdot\right) \rightarrow S_{n-1} \left( \A_0, \cdot\right)$ weakly on $\omega$ as $i \rightarrow \infty$. That is, for any continuous function $f: \omega\rightarrow \mathbb{R}$, one has,
\begin{equation*}
\int_{\omega} f(u)d S_{n-1} \left( \A_i, u\right)  \rightarrow 	\int_{\omega} f(u)dS_{n-1} \left( \A_0, u\right).
\end{equation*}
Moreover, if continuous functions $f_i: \omega\rightarrow \mathbb{R}\,\, (i \in \mathbb{N})$ satisfy that $f_i \rightarrow f$ uniformly on $\omega$, then 
\begin{equation*}
\int_{\omega} f_i(u)d S_{n-1} \left( \A_i, u\right)  \rightarrow\int_{\omega} f(u)d S_{n-1} \left( \A_0, u\right).
\end{equation*}
\end{lemma}

In fact, the weak convergence of the $L_p$ surface area measure of $C$-determined sets by some compact set $\omega\subset \Oc$ can be proved as well. Recall that $\mathbf{C}(\omega)$ denotes the family of continuous functions defined on $\omega$.
 
\begin{lemma}\label{weak-convergence-Sp} Let $p \in \R$ and $\omega\subset\Oc$ be a compact set. Suppose that $\{\A_i\}_{i\in \N_0} \subset \K$ such that $\A_i \to \A_0$ as $i\to \infty$. Then,  $S_{n-1, p}(\A, \cdot) \to S_{n-1, p}(\A_0, \cdot)$ weakly on $\omega$.
\end{lemma}
\begin{proof}
We are required to show that $$\int_{\omega}g(u)dS_{n-1,p} (\A_i, u)\rightarrow \int_{\omega}g(u)dS_{n-1, p} (\A_0, u) \ \ \mathrm{for\ all}\ \ g\in \mathbf{C}(\omega).$$ To this end, as $\A_i \rightarrow \A_0$ and $\{\A_i\}_{i \in \mathbb{N}_0} \subset \mathscr{K}(C, \omega)$, then $h_C(\A_i, \cdot) \rightarrow h_C(\A_0, \cdot)$ uniformly on $\omega$ by \eqref{unif}. Note that $h_C(\A_0, \cdot)$ is continuous and strictly negative on $\omega.$ Then, for any $p\in \R$,  $[-h_C(\A_0, \cdot)]^{1-p}$ is finite and bounded from below  on $\omega$ by a positive number. Thus,  $[-h_C(\A_i, \cdot)]^{1-p} \to [-h_C(\A_0, \cdot)]^{1-p}$ uniformly on $\omega$. As $g\in \mathbf{C}(\omega)$, one has $$g(\cdot)[-h_C(\A_i, \cdot)]^{1-p} \to g(\cdot)[-h_C(\A_0, \cdot)]^{1-p}$$ uniformly on $\omega$. By Lemma \ref{L2}, one has 
\begin{equation*}
\int_{\omega}g(u)[-h_C(\A_i, u)]^{1-p}dS_{n-1} (\A_i, u)\rightarrow \int_{\omega}g(u)[-h_C(\A_0, u)]^{1-p}dS_{n-1} (\A_0, u).
\end{equation*}
Combining with \eqref{Sp}, one can get the desired result.
\end{proof}

The following result provides a solution to the $L_p$ Minkowski problem (i.e., Problem \ref{Lp-Minkowski-unbounded}) when the Borel measure $\mu$ defined on $\Oc$ is nonzero and finite, with its support concentrated on a compact subset of $\Oc$. For $p \notin \{0, 1, n\}$, Theorem \ref{LpCfull} was proven in \cite[Theorems 6.1 and 6.2]{YYZ22}. For $p=1$, Theorem \ref{LpCfull}, including both the existence and uniqueness of the solutions, was proven in \cite{Sch18}. For $p=0$, the existence and uniqueness of solutions were proven in \cite{Sch18} and \cite{YYZ22}, respectively. Let $\omega \subset \Oc$ be a compact set and let $\mathscr{L} = \{ Q \in \mathscr{K}(C, \omega) : V_n(C \setminus Q) = 1 \}$.
 
\begin{theorem}\label{LpCfull}
Let $\mu$ be a nonzero finite Borel measure defined on $\Omega_{C^\circ}$ whose support is concentrated on a compact set $\omega\subset \Omega_{C^\circ}$. Then, for real number 
$p\neq n$, there exists $\A\in \mathcal{K}(C, \omega)$ such that
\begin{equation*}
\mu={S}_{n-1, p}(\A, \cdot).
\end{equation*} Moreover, there exists $\A_0\in  \mathscr{L}$ such that $\A=c^{\frac{1}{n-p}}\A_0$ with 
\begin{equation*}
 c =  \frac{1}{nV_n(C\,\backslash \A_0)}\int_{\omega} (-{h}_C(\A_0, u))^p d\mu(u). 
\end{equation*}   If in addition $0\leq p\leq 1$, then $\A$ is uniquely determined. 
\end{theorem}

\section{Existence of solutions to the \texorpdfstring{$L_p$}{} Minkowski problem for \texorpdfstring{$C$}{}-close sets} \label{sec-3}
\setcounter{equation}{0}
 In this section, we will establish the existence of solutions to Problem \ref{Lp-Minkowski-unbounded} for $\mu$ being a nonzero finite Borel measure on $\Oc.$ Our main result in this section is summarized in the following theorem. 
\begin{theorem}\label{ex&uniC-close}
Let  $0\leq p \leq 1$ and $\mu$ be a nonzero finite Borel measure defined on $\Omega_{C^\circ}$. There exists a unique $C$-close set $\A$ such that $\mu=S_{n-1, p}(\A, \cdot)$.  
\end{theorem} 
Note that Theorem \ref{ex&uniC-close} was proved by Schneider in \cite{Sch18} for $p=0$ and  in \cite{Sch21} for $p=1$. We will use the approximation techniques in \cite{Sch18, Sch21} to prove Theorem \ref{ex&uniC-close}. To this end, we need some preparation.  Recall that  $\xi \in \Omega_{C}$ is fixed, such that, for all  $t>0$,  $$C_t=C\cap H^-_{t}=C \cap H^-(\xi, t)$$ is bounded. Let $b(\A)$ be the distance from the origin to the $C$-compatible set $\A$, given by  
\begin{equation*}
b(\A)=\min \{r>0: rB_n \cap \A\neq \emptyset\}.
\end{equation*}  
Let $\angle (u, \widetilde{u})$ be the angle between $u$ and $\widetilde{u}$, and $\delta_C(u)=\min \{\angle (u, \widetilde{u}):\ \widetilde{u} \in \partial \Oc\}$ be the spherical distance of $u\in \Oc$ to $\partial \Oc$. For $\tau>0$, let  
\begin{align}
\label{def-omega-tau-1} \overline{\omega}(\tau)=\{u \in \Oc: \delta_C(u) \geq \tau\}.
\end{align}  

We shall need the following result regarding the bounds of $b(\A)$.  Its proof is similar to that in \cite[Lemma 9]{Sch24}.  

\begin{lemma}\label{b_0<p<1} 
 Let  $0\leq p\leq 1$, 
$\mu$ be a nonzero finite Borel measure defined on $\Oc$, and  $\A$ be a $C$-close set such that $\mu=S_{n-1, p}(\A, \cdot)$. Let $\tau>0$ be such that $\overline{\omega}(\tau) \subset \Oc$ and 
\begin{equation}\label{low-c-cc}
S_{n-1, p}(\A, \overline{\omega}(\tau))=s_0>0.
\end{equation} 
Then, one can find positive finite constants $\beta_2$ and $\beta_3$, depending on $C$ and $\tau$ only, such that
\begin{equation*}
\beta_2\leq b(\A) \leq \beta_3.
\end{equation*}
\end{lemma} 

\begin{proof} 
We consider the upper bound of $b(\A)$ first.   
For $0\leq p\leq 1$, by \eqref{Sp} and \eqref{hb}, one has
\begin{align*}
\mu(\Oc)&=S_{n-1, p}(\A, \Oc)\\
&= \int_{\Oc} [-h_C(\A, u)]^{1-p}dS_{n-1}(\A, u)\\
&=\int_{\Oc} [-h_C(\A, u)]^{-p} [-h_C(\A, u)] dS_{n-1}(\A, u)\\
&\geq b(\A)^{-p}  \int_{\Oc} [-h_C(\A, u)] dS_{n-1}(\A, u)\\
&=b(\A)^{-p} \cdot n V_n(C\,\backslash \A).
\end{align*}
Together with \eqref{hb} and  the fact that 
$nV_n(C\,\backslash \A)= \int_{\On} \rho_C(\A, v)^ndv$  (see \cite[page 14]{LYZ23}), one has, 
\begin{align} \label{unif-upp-01}
\mu(\Oc)\geq  b(\A)^{-p} \cdot \int_{\On} \rho_C(\A, v)^ndv\geq b(\A)^{n-p} \int_{\On} dv.
\end{align}
Therefore, for $0\leq p\leq 1$, 
\begin{equation*}
b(\A)\leq \left(\frac{ \int_{\On} dv }{\mu(\Oc)} \right)^{\frac{1}{p-n}}=\beta_3,
\end{equation*}
where $\beta_3$ depends on $C$ only. 
 
For the lower bound, note that $b(b(\A)^{-1}\A)=1$ for any $C$-compatible set $\A$. Without loss of generality, suppose that $b(\A)=1$. Thus, there is a point $y \in \partial \A \cap B_n$. For any $x \in \pmb \nu^{-1} (\overline{\omega}(\tau))$,  we have 
$\pmb \nu_{\A}(x) \cap \overline{\omega}(\tau) \neq \emptyset$, and there exists $u \in \overline{\omega}(\tau)$ such that $h_C(\A, u)= \langle x, u \rangle$. Thus, the supporting hyperplane $H(u, h_C(\A, u))$ of $\A$ at $x$ satisfies $
\A \subset H^{-}(u, h_C(\A, u))$, and it separates the origin $o$ and $y$ (as $y \in \A$); see
Figure \ref{fig1}. 
\begin{figure}[ht]
\centering    \includegraphics[width=0.6\textwidth]{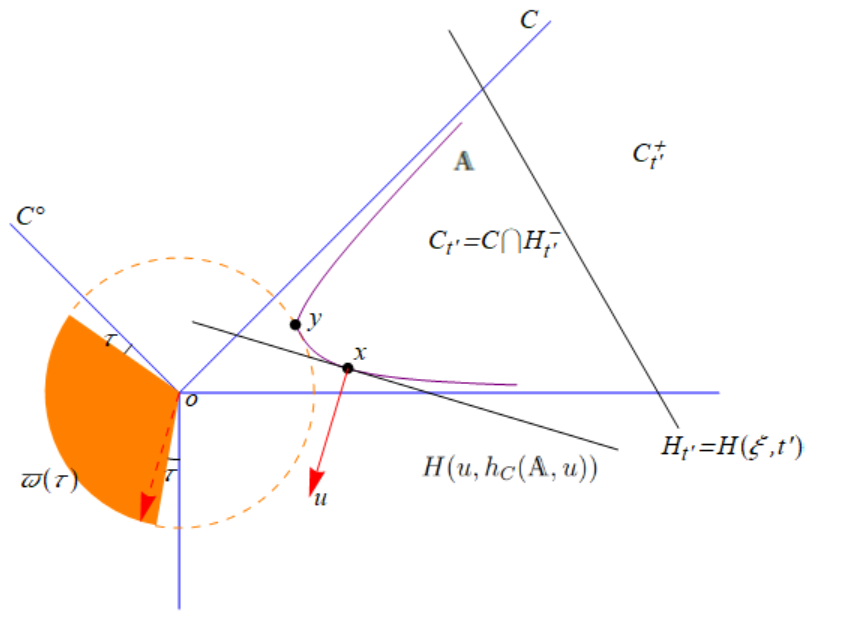}
\caption{}
\label{fig1}
\end{figure}

For any $u \in \overline{\omega}(\tau)$, one has
$h_C(\A, u)\geq \langle y, u \rangle \geq -1$. Define the $C$-determined set (by compact set $ \overline{\omega}(\tau)$) 
\begin{equation*}
\B=C \cap \bigcap_{u \in  \overline{\omega}(\tau)} \{z \in \R^n: \langle z, u \rangle \leq -1\}.
\end{equation*} 
Thus, by \cite[Lemma 9]{Sch24}, 
there exists $t'>0$ such that $\pmb\nu^{-1}_\A(\overline{\omega}(\tau)) \subset C_{t'}$ where $t'$ depends only on $C$ and $\tau$. It follows from the monotonicity of the surface area that  $S_{n-1}(\A, \overline{\omega}(\tau)) \leq S(C_{t'}),$ with $S(C_{t'})=S_{n-1}(C_{t'}, \sphere)$ being the surface area of the convex body $C_{t'}$. Together with \eqref{Sp} and \eqref{hb}, one has, if $b(\A)=1$ and $p\leq 1,$
\begin{align*}
S_{n-1, p}(\A, \overline{\omega}(\tau))
&=\int_{\overline{\omega}(\tau)} [-h_C(\A, u)]^{1-p} dS_{n-1}(\A, u)\\
& \leq \int_{\overline{\omega}(\tau)} b(\A)^{1-p} dS_{n-1}(\A, u)\\
&= S_{n-1}(\A, \overline{\omega}(\tau)) \leq S(C_{t'}).
\end{align*}
Note that the measure $S_{n-1, p}(\A, \cdot)$ has homogeneity on $\A$ as follows: $$S_{n-1, p}(t\A, \cdot)=t^{n-p} S_{n-1, p}(\A, \cdot).$$ For arbitrary $C$-compatible set $\A$ satisfying \eqref{low-c-cc}, due to $b(b(\A)^{-1}\A)=1$, one has
\begin{align} \label{Lp-surface-addition-1}
S(C_{t'}) \geq S_{n-1, p}(b(\A)^{-1}\A, \overline{\omega}(\tau)) =b(\A)^{p-n}S_{n-1, p}(\A, \overline{\omega}(\tau)) =b(\A)^{p-n}\cdot s_0.
\end{align}
Thus, for $p\leq 1$, 
\begin{equation} \label{formula-s0-1}
b(\A)\geq \left( \frac{s_0}{S(C_{t'})} \right)^{\frac{1}{n-p}}=\beta_2, 
\end{equation}
where $\beta_2$ depends on $C$ and $\tau$ only.
\end{proof}

We now prove Theorem \ref{ex&uniC-close} (only for $0<p<1$).  The proof is similar to those   in \cite{LYZ23,Sch18,Sch21}. For the sake of completeness, we include the proof here, but we will keep it brief.

\begin{proof}[Proof of Theorem \ref{ex&uniC-close}.] 
Let $\mu$ be a nonzero finite Borel measure defined on $\Oc,$ and let $\tau>0$ be such that $\mu(\overline{\omega}(\tau))>0.$  For $j\in \mathbb{N}$, define the measure $\mu_j$ by  
\begin{align} \label{definition-nu-j} 
\mu_j(\eta)=\mu(\omega_j \cap \eta)
\end{align}   
for each Borel set $\eta\subseteq \Oc$ with  $\{\omega_j\}_{j \in \mathbb{N}}$ being a sequence of open subsets of $\Oc$ satisfying \begin{align*}
\overline{\omega}(\tau) \subset \omega_1, \ \omega_j \subset \overline{\omega_j}\subset \omega_{j+1}\subset \Omega_{C^\circ}, \ \mathrm{and} \ \cup_{j \in \mathbb{N}} {\omega_j}=\Omega_{C^\circ}. 
\end{align*}
In particular, each $\mu_j$ is a nonzero finite Borel measure on $\Oc$ concentrated on a  compact set $\overline{\omega_j}$. Applying Theorem \ref{LpCfull} to $\mu_j$, a unique $C$-full set $\A_{\omega_j} \in \mathscr{K}(C, \overline{\omega_j})$ can be found such that,  for $0\leq p \leq 1$,  \begin{align}\label{relation-A-omega-j} 
d\mu_j=dS_{n-1, p}(\A_{\omega_j}, \cdot)=(-{h}_C(\A_{\omega_j}, \cdot))^{1-p} d{S}_{n-1} (\A_{\omega_j}, \cdot).
\end{align} 
Moreover, for each $j\in \mathbb{N},$
\begin{equation*}
S_{n-1, p}(\A_{\omega_j}, \overline{\omega}(\tau))=\mu_j(\overline{\omega}(\tau))=\mu(\omega_j \cap \overline{\omega}(\tau))=\mu(\overline{\omega}(\tau))=s_0>0.
\end{equation*} 
It follows from  Lemma \ref{b_0<p<1} that there exist two constants $\overline{\beta}_2$ and $\overline{\beta}_3$, depending on $C$ and $\tau$ only, such that, for  $j \in \mathbb{N}$,
\begin{equation}\label{seq-1-1}
\overline{\beta}_2 \leq b(\A_{\omega_j}) \leq \overline{\beta}_3.
\end{equation} 
It further yields the following fact:  if  $\{t_k\}_{k \in \mathbb{N}}$ is an increasing sequence diverging to $\infty$  with $t_1 > \overline{\beta}_3$, then, for $0< p<1$ and  for $j, k \in \mathbb{N}$, one has  
\begin{equation} \label{ctk}
\emptyset\neq \A_{\omega_j} \cap C_{t_k} \subset C_{t_k}.
\end{equation}   
By the Blaschke selection theorem for $C$-compatible sets (i.e., Theorem \ref{Blas-unb-1}), one can find a convergent subsequence, which will still be denoted by  $\{\A_{\omega_j}\}_{j\in \mathbb{N}}$ for convenience, such that  $\A_{\omega_j}\rightarrow \A$, where $\A$ is a $C$-compatible set. Note that $\A=\cup_{k \in \mathbb{N}}\, M_k$, where for each $k\in \mathbb{N}$, $M_k=\A\cap C_{t_k}$ is a convex body in $\Rn.$ Moreover, $M_k\subset M_{k+1}$ for all $k\in \mathbb{N}$ and $\A_{\omega_j} \cap C_{t_k} \rightarrow M_k$ as $j \rightarrow \infty$. In particular,
\begin{align} \label{conv-supp-c-full} 
h(\A_{\omega_j} \cap C_{t_k}, \cdot) \rightarrow h(M_k, \cdot) \ \ \mathrm{uniformly\ on} \ \ \sphere.
\end{align}

We now claim that $\A=\cup_{k \in \mathbb{N}}\, M_k$ is the desired solution to Problem \ref{Lp-Minkowski-unbounded}.  To this end, we first prove that $\A$ is a $C$-close set, i.e., $V_n(C\,\backslash \A)<\infty$. Since $S_{n-1}(Q, \Oc\,\backslash \omega)=0$ for $Q \in \K$ from \cite[(29)]{Sch18}, by \eqref{Sp} and \eqref{Vn}, one has 
\begin{align*}
V_n(C\,\backslash \A_{\omega_j})
&=\frac{1}{n}\int_{\overline{\omega_j}} [-h_C(\A_{\omega_j}, u)]dS_{n-1}(\A_{\omega_j}, u)\\
&=\frac{1}{n}\int_{\overline{\omega_j}} [-h_C(\A_{\omega_j}, u)]^{p}dS_{n-1, p}(\A_{\omega_j}, u)\\
&=\frac{1}{n}\int_{\overline{\omega_j}} [-h_C(\A_{\omega_j}, u)]^{p}d\mu_j(u).
\end{align*}
Together with \eqref{seq-1-1}, for $0<p<1$, it follows that 
\begin{equation*}
V_n(C\,\backslash \A_{\omega_j}) \leq \frac{\overline{\beta}_3^p}{n} \mu(\Oc)<\infty.
\end{equation*}
Since $\A_{\omega_j} \to \A$ as $j \to \infty$, by \cite[Lemma 2]{Sch24}, one can get that
\begin{align*}
V_n(C\,\backslash \A)\leq \liminf_{j \to \infty} V_n(C\,\backslash \A_{\omega_j}) <\infty.
\end{align*}
Note that $o\notin \A$. To see this, let $\{t_k\}_{k \in \N}$ be as in \eqref{ctk}. Then, for fixed $k\in \N$, 
\begin{align*}
\A_{\omega_j} \cap C_{t_k} \subset C_{t_k} \backslash C_{\overline{\beta}_2/2} \subset C_{t_k} \backslash C_{\overline{\beta}_2/3}.
\end{align*}
Together with $\A_{\omega_j} \to \A$ as $j \to \infty$, one has 
\begin{equation*}
\A \cap C_{t_k} \subset C_{t_k} \backslash C_{\overline{\beta}_2/2} \subset C_{t_k} \backslash C_{\overline{\beta}_2/3}.
\end{equation*}
Thus, $b(\A)>\overline{\beta}_2/3$ and $\A$ is a $C$-close set. 

Let us now prove that $\mu=S_{n-1,p}(\A, \cdot)$ for $0<p<1$. To this end, for any open set $\omega$ such that $\omega \subset \overline{\omega} \subset \Omega_{C^\circ}$, there exists  $j_0 \in \mathbb{N}$ such that $\omega \subset \overline{\omega} \subset \omega_{j_0}$. As $\A=\cup_{k\in \mathbb{N}}M_k,$ there exists $k_0 \in \mathbb{N}$ such that $\pmb\nu^{-1}_{\A}(\omega_{j_0})=\pmb\nu^{-1}_{M_{k}}(\omega_{j_0})$ for all $k\geq k_0$. For $k\geq k_0$ big enough, let 
\begin{equation*}
\widetilde{\A}_j=[C, \overline{\omega_{j_0}}, 
-h(\A_{\omega_j} \cap C_{t_k}, \cdot)] \in \mathcal{K}(C, \overline{\omega_{j_0}}) \ \ \ \text{and} \ \ \ 
\widetilde{\A}_0=[C, \overline{\omega_{j_0}}, -h(M_K , \cdot)] \in \mathcal{K}(C, \overline{\omega_{j_0}}). 
\end{equation*} 
It follows from \eqref{conv-supp-c-full}, and Lemmas \ref{L56} and \ref{weak-convergence-Sp} that $S_{n-1, p}(\widetilde{\A}_j, \cdot) \to S_{n-1, p}(\widetilde{\A}_0, \cdot)$ weakly on $\overline{\omega_{j_0}}$ for $0<p<1$. By 
\cite[Theorem 4.5.1]{Ash72}, for $0<p<1$, 
\begin{align}\label{wco}
S_{n-1, p}(\A, \omega) \notag
&=S_{n-1,p}(\widetilde{\A}_0, \omega)\leq \liminf_{j \rightarrow \infty} S_{n-1,p}(\widetilde{\A}_j, \omega)\\
&=\liminf_{j \rightarrow \infty}S_{n-1, p}(\A_{\omega_j}, \omega)=\liminf_{j \rightarrow \infty}\mu_j(\omega)=\mu({\omega}).
\end{align}
By taking a decreasing sequence of open sets converging to a closed set $\beta\subset \omega_{j_0}$, \eqref{wco} further implies \begin{equation}\label{wcoo}
{S}_{n-1, p}(\A, \beta) \leq \mu(\beta).
\end{equation} 
Similarly, \cite[Theorem 4.5.1]{Ash72} yields 
\begin{align*}
S_{n-1, p}(\A, \beta)
&=S_{n-1,p}(\widetilde{\A}_0, \beta)\geq \limsup_{j \rightarrow \infty} S_{n-1,p}(\widetilde{\A}_j, \beta)\\
&=\limsup_{j \rightarrow \infty}S_{n-1, p}(\A_{\omega_j}, \beta)=\limsup_{j \to \infty}\mu_j(\beta)=\mu({\beta}), 
\end{align*}
for any closed set $\beta \subset \omega_{j_0}$. Together with \eqref{wcoo}, one gets ${S}_{n-1, p}(\A, \beta)=\mu({\beta})$ for any closed set $\beta \subset \omega_{j_0}$. Consequently,   ${S}_{n-1, p}(\A, \beta)=\mu({\beta})$ for every closed set $\beta \subset \Omega_{C^\circ}$, and then, for $0< p<1$, ${S}_{n-1, p}(\A, \cdot)=\mu$ on $\Omega_{C^\circ}$. Hence, $\A$ is a $C$-close set solving Problem \ref{Lp-Minkowski-unbounded} under the condition that $\mu$ is a finite nonzero Borel measure on $\Oc.$

The uniqueness of the solution to Problem \ref{Lp-Minkowski-unbounded} indeed follows from \cite[Remark 5.1]{YYZ22} and here we make it explicit. Let $\A_1, \A_2$ be two $C$-close sets with   $A_1=C\,\backslash\A_1$ and $A_2=C\,\backslash\A_2$. Define 
\begin{equation} \label{Lp-mixed-volume}
\overline{V}_p(A_1, A_2)=\frac{1}{n} \int_{\Oc} \overline{h}_C(A_2, u)^p \overline{h}_C(A_1, u)^{1-p} d\overline{S}_{n-1}(A_1, u).
\end{equation} 
It follows from \eqref{Vn} and H\"older's inequality that $d\nu=\frac{\overline{h}(A_1, \cdot)}{nV_n(A_1)}d\overline{S}_{n-1}(A_1, \cdot)$ is a probability measure on $\Oc$, and  for $0<p<1$, 
\begin{equation}\label{Lp-mixed-V}
\frac{\overline{V}_p(A_1, A_2)}{V_n(A_1)}\leq \left( \frac{\overline{V}_1(A_1, A_2)}{V_n(A_1)}\right)^p \leq \left( \frac{V_n(A_2)}{V_n(A_1)}\right)^{\frac{p}{n}},
\end{equation}  with equality if and only if $A_1=\alpha A_2$ for some $\alpha>0$, where we have used the Minkowski inequality for $C$-coconvex sets established in \cite[(27)]{Sch18}. 
If $S_{n-1, p}(\A_1, \cdot)=S_{n-1, p}(\A_2, \cdot)$ for $0<p<1$, i.e., $\overline{S}_{n-1, p}(A_1, \cdot)=\overline{S}_{n-1, p}(A_2, \cdot)$, it follows from \eqref{volume-Lp} and  \eqref{Lp-mixed-volume} that $\overline{V}_p(A_1, A_2)=V_n(A_2)$ and $\overline{V}_p(A_2, A_1)=V_n(A_1)$. Combining with \eqref{Lp-mixed-V}, one gets   $V_n(A_2)\leq V_n(A_1)$ and $V_n(A_1) \leq V_n(A_2)$, which forces $A_1=A_2$ and therefore $\A_1=\A_2$. 
\end{proof} 
 
The proof of Theorem \ref{ex&uniC-close} gives the following result.
\begin{cor}\label{col-Aj-A} Let $\mu$ be a nonzero finite Borel measure defined on $\Oc$. Let $\mu_j$ and $\A_{\omega_j}$ be as defined in \eqref{definition-nu-j} and, respectively, \eqref{relation-A-omega-j}. Then,  $\A_{\omega_j}\rightarrow \A$ with $\mu=S_{n-1, p}(\A, \cdot).$  \end{cor}
\begin{proof} Let $\{\A_{\omega_{j_k}}\}_{k\in \N}$ be any subsequence of $\{\A_{\omega_j}\}_{j\in \N}$. It follows from  \eqref{seq-1-1} that, for all $k\in \N$,   \begin{equation*} 
\overline{\beta}_2 \leq b(\A_{\omega_{j_k}}) \leq \overline{\beta}_3. 
\end{equation*} By the Blaschke selection theorem for $C$-compatible sets (i.e., Theorem \ref{Blas-unb-1}), one can find a subsequence of $\{\A_{\omega_{j_k}}\}_{k\in \N}$, say $\{\A_{\omega_{j_{k_l}}}\}_{l\in \N}$, such that, $\A_{\omega_{j_{k_l}}} \rightarrow \A_0.$ The proof of Theorem \ref{ex&uniC-close} yields $\mu=S_{n-1, p}(\A_0, \cdot)$ on $\Oc$. Note that $\mu=S_{n-1, p}(\A,\cdot)$ as well, and hence $\A=\A_0$ following from  the uniqueness in Theorem \ref{ex&uniC-close}. 

In summary, we have proved that any subsequence of $\{\A_{\omega_j}\}_{j\in \N_0}$ admits a convergent subsequence converging to $\A$. This proves that $\A_{\omega_j}\rightarrow \A$ as desired.  
\end{proof}
  
Bearing in mind that solutions to Problem \ref{Lp-Minkowski-unbounded} are unique, it is natural to study the continuity of the solutions. We are interested in the following problem.   
\begin{problem}[\bf The continuity of the solutions to $L_p$ Minkowski problem for $C$-close sets]\label{con?} Let $\nu_i$ for $i\in \mathbb{N}_0$ be nonzero finite Borel measures on $\Oc$ and $0\leq p\leq 1$. For $i\in \mathbb{N}_0,$ let  $\B_i$ be the unique $C$-close set such that $ \nu_i=S_{n-1, p}(\B_i, \cdot)$.  Under what conditions on $\{\nu_i\}_{i\in \N_0}$ does $\B_i \rightarrow \B_0$?
\end{problem}

\section{Continuity if \texorpdfstring{$\nu_i$}{}  concentrate on a common compact set } \label{sec4.1}\setcounter{equation}{0}
In this section, we will discuss the case when all measures  $\nu_i$, for $i\in \N_0$, concentrate on a common compact set $\omega\subset\Oc$. 
For $0\leq p\leq 1$, let $\nu_i$ and $\B_i$ be as in Problem \ref{con?}.  By Theorem \ref{LpCfull}, each $\B_i\in \mathcal{K}(C, \omega)$ for $i\in \N_0.$ Moreover,  
$\nu_i=S_{n-1, p}(\B_i, \cdot)$ where $\B_i=c_i^{\frac{1}{n-p}}\B_{i}^0$
with  
\begin{align}\label{Ai-compact-support}
c_i&= \frac{1}{n}\int_{\omega} (-{h}_C(\B_{i}^0, u))^p d\nu_i(u), \\ 
\B_{i}^0 & \in \mathscr{L}=\left\{ Q \in \mathscr{K}(C, \omega):\,\, V_n(C\, \backslash\, Q)=1 \right\}.  \notag
\end{align} 
Note that for each $Q\in \mathcal{L}$, $(-{h}_C(Q, \cdot))^p>0$ on $\omega$ following from the facts that $h_C(Q, \cdot)$ is continuous and negative on $\omega$. If $\nu$ is a nonzero finite Borel measure concentrated on $\omega$, one sees that $$\int_{\omega} (-{h}_C(Q, u))^p d\nu(u)>0.$$ So each $c_i$ in  \eqref{Ai-compact-support} must be strictly positive.   

We will prove the following result. 
\begin{theorem}\label{m1}
Let $0\leq p\leq 1$ and $\omega\subset\Oc$ be a compact subset. 
Let $\nu_i$ for $i\in \mathbb{N}_0$ be nonzero finite Borel measures on $\Oc$ whose supports are contained in $\omega$. Let $\B_i\in \K$ be
such that $\nu_i=S_{n-1, p}(\B_i, \cdot)$ for $i\in \mathbb{N}_0.$  If $\nu_i \rightarrow \nu_0$ weakly on $\omega$, then $\B_i \rightarrow \B_0$ as $i\rightarrow \infty$.  
\end{theorem}
 
To prove Theorem \ref{m1}, we need the following lemma, whose proof is inspired by Lemma \ref{L1}.  
\begin{lemma} \label{bound_Q} 
Let $\omega\subset \Oc$ be a given compact set. Then, there exist positive constants $\beta_4$ and $\beta_5$, depending on $C$ and $\omega$ only, such that 
\begin{equation*}
\beta_4 \leq b(Q) \leq \beta_5\ \  \mathrm{for\ each} \ \ Q\in \mathscr{L}.
\end{equation*} 
\end{lemma} 
\begin{proof} 
We consider the lower bound first, which can be obtained similar to the proof of Lemma \ref{b_0<p<1}. Let $Q \in \mathscr{L}$.  Lemma \ref{L1} implies the existence of constant $\widetilde{t}_0>0$ depending on $C$ and $\omega$ only, such that $C\, \cap\,H_{\widetilde{t}_0}\subset Q $. This implies $\pmb\nu_{Q}^{-1}(\omega) \subset C_{\widetilde{t}_0}$ and then $S_{n-1}(Q, \omega) \leq S(C_{\widetilde{t}_0})$, where $S(C_{\widetilde{t}_0})$ is the surface area of the convex body $C_{\widetilde{t}_0}$. By \eqref{hb}, \eqref{Vn} and the fact that $S_{n-1}(Q, \Oc\backslash \omega)=0$ (see \cite[(29)]{Sch18}), one has, 
\begin{equation*}
1=V_n(C\,\backslash Q)=\frac{1}{n}\int_{\omega} (-{h}_C(Q, u))dS_{n-1}(Q, u)\leq \frac{1}{n}\, b(Q) \cdot S(C_{\widetilde{t}_0}).
\end{equation*} This further gives   
\begin{equation*}
b(Q)\geq   \frac{n}{S(C_{\widetilde{t}_0})}=\beta_4,  
\end{equation*} and $\beta_4$ depends on $C$ and $\omega$ only.

Let us prove the upper bound.  
By \cite[page 14]{LYZ23} and \eqref{hb},  for $Q\in \mathcal{L}$, one has 
\begin{equation*}
1=V_n(C\setminus Q)=\frac{1}{n} \int_{\On} \rho_C(Q, v)^n dv \geq \frac{b(Q)^n}{n} \int_{\On} dv.
\end{equation*} 
 After a rearrangement, one gets  \begin{equation*}
b(Q)\leq \bigg( \frac{n}{\int_{\On} dv}\bigg)^\frac{1}{n}=\beta_5,
\end{equation*}
where $\beta_5$ depends on $C$. This completes  the proof.
\end{proof}

With the help of the Blaschke selection theorem for $C$-compatible sets (i.e., Theorem \ref{Blas-unb-1}),   we can obtain the following result.
\begin{lemma}\label{L2.2} 
Let $\omega\subset\Oc$ be a given compact set. Every sequence $\{Q_i\}_{i \in \mathbb{N}}$   with $Q_i \in \mathscr{L}$ for $i\in \N$ must admit a convergent subsequence whose limit is also in $\mathcal{L}$.  
\end{lemma}
\begin{proof} 
Let $\{Q_i\}_{i \in \mathbb{N}}$ be a sequence  with $Q_i \in \mathscr{L}$ for $i\in \N$. By Lemma \ref{bound_Q}, one gets  $\beta_4 \leq b(Q_i) \leq \beta_5$  for each $i\in \N$. It follows from   Theorem \ref{Blas-unb-1}  that the sequence $\{Q_i\}_{i \in \mathbb{N}}$ has a subsequence, say $\{Q_{i_j}\}_{j \in \mathbb{N}}$, such that $Q_{i_j}\rightarrow Q_0$, where $Q_0$ is a $C$-compatible set. 

We now claim that $C\cap H_{\widetilde{t}_0}\subset Q_0$, where $\widetilde{t}_0$ is the constant given in Lemma \ref{L1}. In fact, it follows from Lemma \ref{L1}  that $C \cap H_{\widetilde{t}_0} \subset Q_i$ for all $i \in \mathbb{N}$. In particular, $C \cap H_{\widetilde{t}_0} \subset Q_{i_j}$ for each $j \in \mathbb{N}$. Taking the limit as $j\rightarrow \infty,$ one gets $C \cap H_{\widetilde{t}_0} \subset Q_{0}.$  This yields that $Q_0$ is a $C$-full set. Applying Lemma \ref{L56}, one sees that $Q_0\in \K$. In summary, $Q_{i_j}\rightarrow Q_0 \in \K$ as $j\to \infty$.  
Note that $S_{n-1}(Q, \Oc\backslash \omega)=0$ for $Q \in \K$.
It follows from \eqref{unif}, \eqref{Vn} and  Lemma \ref{L2} that 
\begin{align*}
\lim_{j\rightarrow \infty} V_n(C\setminus Q_{i_j}) 
&=\lim_{j\rightarrow \infty} \frac{1}{n} \int_{\omega} (-h_C(Q_{i_j}, u))dS_{n-1}(Q_{i_j}, u)\\
&=\frac{1}{n} \int_{\omega} (-h_C(Q_{0}, u))dS_{n-1}(Q_{0}, u)\\
&=V_n(C\setminus Q_0). 
\end{align*}  
That is, $V_n(C\setminus Q_0)=1$ and $Q_0 \in \mathcal{L}$.   
\end{proof}

We are now in the position to prove Theorem \ref{m1}. 

\begin{proof}[Proof of Theorem \ref{m1}.]  
Let $0\leq p\leq 1$ and $\omega\subset\Oc$ be a compact subset. 
Let $\nu_i$ for $i\in \mathbb{N}_0$ be nonzero finite Borel measures on $\Oc$ whose supports are contained in $\omega$ and such that $\nu_i \to \nu_0$ weakly on $\omega$. Let $\B_i\in \K$ be such that $\nu_i=S_{n-1, p}(\B_i, \cdot)$.

We want to show $\B_i \rightarrow \B_0$ as $i \rightarrow \infty$. To this end,  it suffices to prove that every subsequence of $\{\B_i\}_{i \in \mathbb{N}}$ has a convergent subsequence whose limit must be $\B_0$. Recall that $$\mathscr{L}=\left\{ Q \in \mathscr{K}(C, \omega):\,\, V_n(C\, \backslash\, Q)=1 \right\},$$ and $\B_i=c^\frac{1}{n-p}_i \B^0_i$ with $c_i$ the constant given in  \eqref{Ai-compact-support} and $\B_i^0\in \mathcal{L}.$

It follows from Lemma \ref{L2.2}  that any subsequence of $\{\B_i^0 \}_{i\in \N}$, say $\{\B_{i_j}^0 \}_{j\in \N}$, admits a convergent subsequence $\{\B_{i_{j_k}}^0 \}_{k\in \N}$ such that $\B^0_{i_{j_k}}\rightarrow \widetilde{\B}^0_0 \in \mathcal{L}.$ This further implies $h_C(\B^0_{i_{j_k}}, \cdot) \rightarrow h_C(\widetilde{\B}^0_0, \cdot)$ uniformly on compact set $\omega$ by \eqref{unif}.
Let 
\begin{equation*}
\widetilde{\B}_0=\left[ \frac{1}{n}\int_{\omega} (-h_C(\widetilde{\B}^0_0, u))^p d\nu_0(u)\right]^{\frac{1}{n-p}}\widetilde{\B}^0_0=\widetilde{c_0}^{\frac{1}{n-p}}\widetilde{\B}^0_0.
\end{equation*}

We now claim that $\widetilde{\B}_0=\B_0.$ First of all,  as $\nu_i\rightarrow \nu_0$ weakly on $\omega$ and $h_C(\B^0_{i_{j_k}}, \cdot) \rightarrow h_C(\widetilde{\B}^0_0, \cdot)$ uniformly on compact set $\omega$, Lemma \ref{L2} implies that, for $0\leq p\leq 1$,  
\begin{align*} 
\lim_{k\rightarrow \infty} c_{i_{j_k}} =\lim_{k\rightarrow \infty} \frac{1}{n}\int_{\omega} (-{h}_C(\B^0_{i_{j_k}}, u))^p d\nu_{i_{j_k}}(u) = \frac{1}{n}\int_{\omega} (-h_C(\widetilde{\B}^0_0, u))^p d\nu_0(u)=\widetilde{c_0}>0.  
\end{align*}  In particular, one sees that $\B_{i_{j_k}} \rightarrow \widetilde{\B}_0$ as $k\rightarrow \infty$. Note that $\B_{i_{j_k}}$ for each $k\in \N$ and $\widetilde{\B}_0$ are in $\K.$  It follows from Lemma \ref{weak-convergence-Sp} that, for $0\leq p\leq 1,$ ${S}_{n-1, p}(\B_{i_{j_k}}, \cdot) \rightarrow {S}_{n-1, p}(\widetilde{\B}_0, \cdot)$ weakly on $\omega$ as $k\rightarrow \infty$. Thus, $$ 
\nu_{i_{j_k}}= {S}_{n-1, p}(\B_{i_{j_k}}, \cdot) \rightarrow {S}_{n-1, p}(\widetilde{\B}_0, \cdot). $$ Then $\nu_0={S}_{n-1, p}(\widetilde{\B}_0, \cdot)$ due to the fact that $\nu_{i_{j_k}}\rightarrow \nu_0$ weakly on $\omega.$ Note that $\nu_0={S}_{n-1, p}(\B_0, \cdot)$ and the uniqueness in Theorem \ref{LpCfull} further yields that $\B_0=\widetilde{\B}_0$.  

In summary, we prove that every subsequence of $\{\B_i\}_{i \in \mathbb{N}}$ has a convergent subsequence whose limit must be $\B_0$. Thus, $\B_i\rightarrow \B_0$ as $i\rightarrow \infty.$ This completes the proof.  
\end{proof}

\section{The continuity when \texorpdfstring{$\mu$}{} are finite Borel measures on \texorpdfstring{$\Oc$}{}}  \label{sec-4.12}
\setcounter{equation}{0}
In this section, we assume that $\nu_i$, $i\in \N_0$, are nonzero finite Borel measures on $\Oc$ (not necessarily concentrated on compact sets). It follows from \cite[Theorem 7.8]{Folland-1} that each $\nu_i$ is a regular Radon measure. The measures $\nu_i$, $i\in \mathbb{N}$, are said to converge to $\nu_0$ vaguely on $\Oc$ if 
\begin{equation}\label{vague-convergence}
\int_{\Oc} f(u)d\nu_i(u) \to \int_{\Oc} f(u)d\nu_0(u)\ \ \text{for all}\ \  f\in \mathbf{C}_c(\Oc)
\end{equation}
as $i \to \infty$,
where $\mathbf{C}_c(\Oc)$ is the collection of continuous functions $f: \Oc \to \mathbb{R}$ such that the support of $f$ is a compact subset of $\Oc$. It follows from the Riesz representation theorem (i.e., \cite[Theorem 7.2]{Folland-1}) that the limit under the vague convergence is unique. 
 
The following lemma establishes the vague convergence of the $L_p$ surface area measures of $C$-compatible sets. Recall that $\Ln$ denotes the family of  $C$-compatible sets.  
\begin{lemma} \label{vague-con-Sp}
Let $\{\A_i\}_{i \in \N_0} \subset \Ln$ and let $p\in \R.$ If $\A_i \to \A_0$, then $$S_{n-1, p}(\A_i, \cdot)\to S_{n-1, p}(\A_0, \cdot) \ \ \mathrm{vaguely\  on} \ \ \Oc.$$   
\end{lemma} 
\begin{proof} Let $\{\A_i\}_{i \in \N_0} \subset \Ln$ be such that $\A_i \to \A_0$ and let $p\in \R.$  
By \eqref{vague-convergence}, it suffices to show 
\begin{equation*}
\lim_{i\rightarrow \infty} \int_{\Oc}g(u)dS_{n-1, p} (\A_i, u)= \int_{\Oc}g(u)dS_{n-1, p} (\A_0, u) \ \ \mathrm{for\ all}\ \ g\in \mathbf{C}_c(\Oc).
\end{equation*} To this end, let $g \in \mathbf{C}_c(\Oc)$ with compact support $\omega \subset \Oc$. Let $$\tau=\inf \{\angle (u, v): u \in \omega \,\,\text{and}\,\, v \in \partial \Oc\}.$$ As  $\tau$ is the infimum of a positive continuous function over a compact set $\omega \times \partial \Oc$, one gets  $\tau>0$. 
Recall that $\overline{\omega}(\tau)=\{u \in \Oc: \delta_C(u)\geq \tau\}$. 
It can be observed that $$\omega \subseteq \overline{\omega}(\tau) \subset \overline{\omega}(\tau/2) \subset \Oc.$$ Since $\A_i \to \A_0$, by Definition \ref{con-def}, there is $\beta>0$ such that $b(\A_i) \leq \beta$ for all $i \in \N_0$. Applying \cite[Lemma 7]{Sch24} to each $\A_i$, one has, for each $i \in \N_0$,
\begin{equation*}
\pmb \nu^{-1}_{\A_i}(\omega) \subset \pmb \nu^{-1}_{\A_i} (\overline{\omega}(\tau/2)) \subset \frac{b(\A_i)}{\sin (\tau/2)} B_n \subset \frac{\beta}{\sin (\tau/2)} B_n.
\end{equation*} 
Let $t_\tau>\frac{\beta}{\sin (\tau/2)}$ be a constant, depending only on compact set $\omega$, such that 
\begin{equation*}
\frac{\beta}{\sin (\tau/2)} B_n \cap C \subset C_{t_\tau} =C\cap H^-_{t_\tau}.
\end{equation*} Then,  for each $i\in \N_0$,
\begin{equation} \label{reverse-Gauss-image}
\pmb \nu^{-1}_{\A_i}(\omega) \subset C_{t_\tau} \ \ \mathrm{and} \ \ \pmb \nu^{-1}_{\A_i}(\omega)=\pmb \nu^{-1}_{\A_i \cap C_{t_\tau}}(\omega). 
\end{equation} 
 Together with the fact that $g \in \mathbf{C}_c(\Oc)$ has compact support $\omega \subset \Oc$, this further implies 
\begin{align*}
\int_{\Oc}g(u)dS_{n-1, p} (\A_i, u)=\int_{\omega} g(u)dS_{n-1, p}(\A_i, u)
=\int_{\omega} g(u)dS_{n-1, p}(\A_i \cap C_{t_\tau}, u).
\end{align*}  
Without loss of generality, we can further choose $t_\tau$ big enough so that 
$\A_i\cap C_{t_\tau}$ are convex bodies  for all $i \in \N_0$. Note that $\A_i \cap C_{t_\tau} \to \A_0 \cap C_{t_\tau}$, and thus, $h(\A_i \cap C_{t_\tau}, \cdot) \to h(\A_0 \cap C_{t_\tau}, \cdot)$ uniformly on $\sphere$ (and in particular on $\overline{\omega}(\tau/2)$). Based on \eqref{Wullf-shape}, we can define $C$-determined sets   
\begin{equation}
\widetilde{\A}_i=[C, \overline{\omega}(\tau/2), -h(\A_i \cap C_{t_\tau}, \cdot)] \in \mathcal{K}(C, \overline{\omega}(\tau/2)) \ \  \text{for} \ \ i\in \mathbb{N}_0. \label{construct-1} 
\end{equation} 
It follows from Lemma \ref{L56} that $\widetilde{\A}_i \to \widetilde{\A}_0$, and thus Lemma \ref{weak-convergence-Sp} yields $S_{n-1, p}(\widetilde{\A}_i, \cdot) \to S_{n-1, p}(\widetilde{\A}_0, \cdot)$ weakly on $\overline{\omega}(\tau/2)$ for $p \in \R$. By \eqref{def-weakly}, for every continuous $f: \overline{\omega}(\tau/2) \to \R$, one has
\begin{equation*} \lim_{i\rightarrow \infty}
\int_{\overline{\omega}(\tau/2)}f(u)dS_{n-1, p} (\widetilde{\A}_i, u)= \int_{\overline{\omega}(\tau/2)}f(u)dS_{n-1, p} (\widetilde{\A}_0, u)).    
\end{equation*}In particular, this holds for $g \in \mathbf{C}_c(\Oc)$ as its support $\omega\subset \overline{\omega}(\tau/2)$ is compact in $\Oc.$ That is, \begin{equation} \lim_{i\rightarrow \infty}
\int_{ \omega} g(u)dS_{n-1, p} (\widetilde{\A}_i, u)= \int_{ \omega} g(u)dS_{n-1, p} (\widetilde{\A}_0, u)). \label{limit-construct}    \end{equation} It follows from \eqref{reverse-Gauss-image} and \eqref{construct-1} that, for each $i\in \mathbb{N}_0$, $$\int_{ \omega} g(u)dS_{n-1, p} (\widetilde{\A}_i, u)=\int_{ \omega} g(u)dS_{n-1, p} ( \A_i, u). $$ Together with \eqref{limit-construct}, one gets $$\lim_{i\rightarrow \infty}
\int_{ \omega} g(u)dS_{n-1, p} ( \A_i, u)= \int_{ \omega} g(u)dS_{n-1, p} ( \A_0, u)).$$ The desired vague convergence that $ S_{n-1, p}(\A_i, \cdot)\to S_{n-1, p}(\A_0, \cdot)$ on  $\Oc$ then follows from the  arbitrariness of $g\in \mathbf{C}_c(\Oc)$.  
\end{proof}

   Based on Lemma \ref{vague-con-Sp}, we can get the following continuity result, and hence solve Problem \ref{con?} under the mild conditions on $\nu_i$ for $i\in \mathbb{N}_0$. 

\begin{theorem}\label{con-mu}
Let $0\leq p\leq 1$. Assume that $\nu_i$, $i\in \N_0$ are nonzero finite Borel measures such that  $\nu_i$ converges to $\nu_0$ vaguely on $\Oc$ and  $\sup \{\nu_i(\Oc): i \in \N_0\}<\infty.$   If,  for each $i\in \mathbb{N}_0$,  $\B_i$ is the unique $C$-close set solving Problem \ref{Lp-Minkowski-unbounded} (i.e., $\nu_i=S_{n-1, p}(\B_i, \cdot)$), then $\B_i \to \B_0$ as $i\to \infty$.
\end{theorem}

In order to prove Theorem \ref{con-mu}, the next lemma regarding the uniform upper and lower bounds of $\{b(\B_i)\}_{i \in \N_0}$ is crucial. 

\begin{lemma}\label{widet0}  
Let $0\leq p\leq 1$, and let $\nu_i$, $i\in \N_0$ be nonzero finite Borel measures such that $\nu_i$ converges to $\nu_0$ vaguely on $\Oc$ and $\sup \{\nu_i(\Oc): i \in \N_0\}<\infty.$ Then, there exist two positive constants $\zeta_0$ and $\zeta_1$,  such that, for all $i \in \N_0$,  $$\zeta_0\leq b(\B_i) \leq \zeta_1.$$  
\end{lemma}
\begin{proof} As $\nu_0$ is a nonzero finite Borel measure on $\Oc$, there exists a compact set $\omega_0\subset \Oc$ such that $\nu_0(\omega_0)>0.$   Let $U \subset \Oc$ be an open set satisfying  $\omega_0 \subset U \subset \overline{U}\subset \Oc$. Let $g \in \mathbf{C}_c(\Oc)$ be such that $0\leq g\leq 1$ on $U$, $g\equiv 1$ on $\omega_0$, and the (compact) support of $g$ is contained in $U$. Since $\nu_i \to \nu_0$ vaguely on $\Oc$, by \eqref{vague-convergence}, one has \begin{equation*}
0<\nu_0(\omega_0) \leq \int_{\Oc} g(u)d\nu_0(u) =\lim_{i \to \infty} \int_{\Oc} g(u)d\nu_i(u) \leq \liminf_{i \to \infty} \nu_i(U).
\end{equation*} Thus, for $\epsilon_0= \frac{\nu_0(\omega_0)}{2}>0,$ one can find a constant $i_0=i_0(\epsilon_0)\in \mathbb{N},$ such that $$\inf_{i \geq i_0} \nu_i(U)\geq \nu_0(\omega_0)-\epsilon_0=\frac{\nu_0(\omega_0)}{2}>0. $$ Consequently, $\nu_i(U)\geq \frac{\nu_0(\omega_0)}{2}>0$ for all $i\geq i_0.$ Without loss of generality, we can assume that $$m_0=\inf_{i\geq 1} \nu_i(\overline{U})>0. $$ 

As $\overline{U}\subset \Oc$, one can find a constant $\tau_0>0$ such that $\overline{\omega}(\tau_0)$ (see its definition in \eqref{def-omega-tau-1}) is a compact subset of $\Oc$ with $\overline{U}\subset \overline{\omega}(\tau_0)$.  Employing the first inequality in \eqref{formula-s0-1} to each $\B_i$, one gets,  for $0\leq p \leq 1$,  
\begin{align*}
b(\B_i)\geq \left( \frac{S_{n-1, p}(\B_i, \overline{\omega}(\tau_0))}{S(C_{\widetilde{t'}})} \right)^{\frac{1}{n-p}}=\left( \frac{\nu_i( \overline{\omega}(\tau_0))}{S(C_{\widetilde{t'}})} \right)^{\frac{1}{n-p}} \geq \left( \frac{m_0}{S(C_{\widetilde{t'}})} \right)^{\frac{1}{n-p}},
\end{align*} 
where $\widetilde{t'}$ is a constant depending only on $C$ and $\tau_0$. This gives the lower bound with $$\zeta_0=\left( \frac{m_0}{S(C_{\widetilde{t'}})}\right)^{\frac{1}{n-p}}.$$
 
For the upper bound, let $ 
M_0=\sup \{\nu_i(\Oc): i \in \N_0\}<\infty.$ 
Employing \eqref{unif-upp-01} to each $\B_i$, one gets, for all $i\in \N_0$,   
\begin{align*} 
M_0\geq \nu_i(\Oc)\geq  b(\B_i)^{n-p} \cdot \int_{\On}dv. 
\end{align*} 
This further gives, for all $i \in \N_0$, 
\begin{equation*}
b(\B_i)\leq\bigg(\frac{ \int_{\On}dv }{M_0} \bigg)^{\frac{1}{p-n}}:=\zeta_1. 
\end{equation*} This establishes the desired upper bound. 
\end{proof}  

\begin{proof}[Proof of Theorem \ref{con-mu}.] 
Let $0\leq p\leq 1$ and $i \in \N$. To show $\B_i \to \B_0$, it suffices to show that every subsequence of $\{\B_i\}_{i \in \N}$ has a convergent subsequence with its limit being $\B_0$. Let $\{\B_{i_j}\}_{j \in \N} \subset \{\B_i\}_{i\in \N}$. By Lemma \ref{widet0}, for all $j \in \N$, one has $\zeta_0 \leq b(\B_{i_j}) \leq \zeta_1$. Employing Theorem \ref{Blas-unb-1} to sequence $\{\B_{i_j}\}_{j \in \N}$ yields the existence of a convergent subsequence, namely, $\{\B_{i_{j_k}}\}_{k \in \N} \subset \{\B_{i_j}\}_{j \in \N}$ such that as $k \to \infty$,
\begin{equation*}
\B_{i_{j_k}} \to \widehat{\B}_0,
\end{equation*}
where $\widehat{\B}_0$ is a $C$-compatible set.
Together with Lemma \ref{vague-con-Sp}, one can get that 
\begin{equation*}
\nu_{i_{j_k}}=S_{n-1, p}(\B_{i_{j_k}}, \cdot) \to S_{n-1, p}(\widehat{\B}_0, \cdot),
\end{equation*} vaguely on $\Oc$. Thus, $\nu_0=S_{n-1, p}(\widehat{\B}_0, \cdot)$ because $\nu_{i_{j_k}} \to \nu_0$ vaguely on $\Oc$. Recall that $\nu_0=S_{n-1, p}(\B_0, \cdot)$. The uniqueness in Theorem \ref{ex&uniC-close} forces $\widehat{\B}_0=\B_0$, as desired.
\end{proof}

\section{Continuity with respect to \texorpdfstring{$p\in [0, 1]$}{}} 
\setcounter{equation}{0}
In this section, we aim to address the following problem.
\begin{problem}\label{continuity-Sp2?}
Let $\nu$ be a nonzero finite Borel measure on $\Oc$ and $0\leq p_i\leq 1$ for $i\in \N_0$. Let $\E_i$ be the unique $C$-close set, such that for $i\in \N_0$, 
\begin{equation} \label{continuity-p-i-1}
\nu=S_{n-1, p_i}(\E_i, \cdot).
\end{equation} 
Under what conditions on $\{p_i\}_{i \in \N_0}$ does $\E_i \to \E_0$?
\end{problem} 

We first concentrate on the case when the measure $\nu$ has its support concentrated on a compact set, say $\omega\subset \Oc.$ Let $p_i\in [0, 1]$ for all $i\in \N_0$ and let $\E_i$ be such that $\nu=S_{n-1, p_i}(\E_i, \cdot)$. Note that $\E_i\in \K$ for all $i\in \N_0,$ due to Theorem \ref{LpCfull}. For each $i\in \N_0$,   $\E_i $ is unique and satisfies  
\begin{equation} \label{Aip}
\E_i=\left[ \frac{1}{n}\int_{\omega} (-{h}_C(\E_i^0, u))^{p_i} d\nu(u)  \right]^\frac{1}{n-p_i} \E_i^0=c_i^\frac{1}{n-p_i} \E_i^0,
\end{equation}
with $\E_i^0 \in \mathscr{L}=\left\{ Q \in \mathscr{K}(C, \omega):\,\, V_n(C\, \backslash\, Q)=1 \right\}$. 

We now establish the continuity of the solutions to the $L_p$ Minkowski problem as $p_i\rightarrow p_0.$
  
\begin{theorem}\label{con-C-full-p}
Let $\nu$ be a nonzero finite Borel measure defined on $\Oc$ whose support is concentrated on a compact set $\omega \subset \Oc$. 
Let $i \in \mathbb{N}_0$ and $0\leq p_i\leq 1$ be such that $p_i\rightarrow p_0$. If  $\E_i \in \K $ for each $i\in \N_0$ is the unique solution to \eqref{continuity-p-i-1}, then  $\E_i \rightarrow \E_0$ as $i \rightarrow \infty$.
\end{theorem}
\begin{proof} 
Let $\{\E_{i_j}^0\}_{j\in \mathbb{N}}$ be an arbitrary subsequence of $\{\E_{i}^0\}_{i\in \mathbb{N}}$. The fact that $\{\E_{i_j}^0\}_{j\in \mathbb{N}} \subset \mathscr{L}$ for $j\in \N_0$ implies that, due to Lemma \ref{L2.2}, every subsequence of 
$\{\E_{i_j}^0\}_{j\in \mathbb{N}}$ admits a convergent subsequence whose limit is also in $\mathscr{L}$. Let $\{\E_{i_{j_k}}^0\}_{k \in \mathbb{N}}$ be such a subsequence and let $\E_{i_{j_k}}^0 \to \widetilde{\E}_0^0 \in \mathscr{L}$ as $k \to \infty$. It follows from \eqref{unif} that $h_C(\E_{i_{j_k}}^0, \cdot) \rightarrow h_C(\widetilde{\E}_0^0, \cdot)$ uniformly on compact set $\omega$. Since $0\leq p_i\leq 1$, a positive constant $m_1$ can be found so that  
\begin{equation} \label{uniform-bound-10}
(-h_C(\widetilde{\E}^0_0, u))^{p_0}\leq m_1 \,\,\,\text{and}\,\,\,-{h}_C(\E^0_{i_{j_k}}, u))^{p_{i_{j_k}}}\leq m_1 
\end{equation} 
for all $k\in \N_0$ and $u \in \omega$. Define
\begin{equation*}
\widetilde{\E}_0=\left[\frac{1}{n}\int_{\omega} (-h_C(\widetilde{\E}^0_0, u))^{p_0} d\nu(u)\right]^\frac{1}{n-p_0} \widetilde{\E}_0^0=\widetilde{c}_0^\frac{1}{n-p_0} \widetilde{\E}^0_0.
\end{equation*} 
 
We now claim that $\widetilde{\E}_0=\E_0.$ By  \eqref{Aip} and \eqref{uniform-bound-10}, the dominated convergence theorem implies that, if $p_i\to p_0$, 
\begin{align*} 
\lim_{k\rightarrow \infty} c_{i_{j_k}} =\lim_{k\rightarrow \infty} \frac{1}{n}\int_{\omega} (-{h}_C(\E^0_{i_{j_k}}, u))^{p_{i_{j_k}}} d\nu(u) = \frac{1}{n}\int_{\omega} (-h_C(\widetilde{\E}^0_0, u))^{p_0} d\nu(u)=\widetilde{c}_0>0.
\end{align*} 
This further implies that $\E_{i_{j_k}} \rightarrow \widetilde{\E}_0$ as $i \rightarrow \infty$ due to \eqref{Aip}. Consequently,  $h_C(\E_{i_{j_k}}, \cdot)\rightarrow h_C(\widetilde{\E}_0, \cdot)$ uniformly on $\omega$ by \eqref{unif} and  $S_{n-1}(\E_{i_{j_k}}, \cdot) \rightarrow S_{n-1}(\widetilde{\E}_0, \cdot)$ weakly on $\omega$ from Lemma \ref{L2}.  Thus, the following weak convergence holds, due to \eqref{continuity-p-i-1} and Lemma \ref{L2}, \begin{align*}
d\nu&=dS_{n-1, p_{i_{j_k}}}(\E_{i_{j_k}}, \cdot)\\ &=(-h_C(\E_{i_{j_k}}, \cdot))^{1-p_{i_{j_k}}}d S_{n-1}(\E_{i_{j_k}}, \cdot)  \\ & \rightarrow (-h_C(\widetilde{\E}_0, \cdot))^{1-p_0} dS_{n-1}(\widetilde{\E}_0, \cdot) \\
&=dS_{n-1,p_0}(\widetilde{\E}_0, \cdot).
\end{align*} 
In particular,  one gets  $\nu=S_{n-1,p_0}(\widetilde{\E}_0, \cdot)$. As $\nu=S_{n-1, p_0}(\E_0, \cdot)$, the uniqueness in Theorem \ref{LpCfull} forces $\widetilde{\E}_0=\E_0$,  as desired. 
\end{proof}

Our last result deals with the continuity of solutions to the $L_p$ Minkowski problem as $p_i\rightarrow p_0$ for nonzero finite Borel measure $\nu$ on $\Oc$ whose support may not concentrate on a compact set.  
Let $i \in \mathbb{N}_0$ and $0\leq p_i\leq 1$.
It follows from Theorem \ref{ex&uniC-close} that there exist unique $C$-close sets $\E_i$ so that, for each $i\in \N_0,$
\begin{equation}\label{Ei-Sp}
\nu=S_{n-1, p_i}(\E_i, \cdot).
\end{equation}

We shall need the following lemma. 
 
\begin{lemma}\label{vague-con-Spi}
Let $\{\A_i\}_{i \in \N_0} \subset \Ln$ and $0\leq p_i \leq 1$. If $\A_i \to \A_0$ and $p_i \to p_0$, then $S_{n-1, p_i}(\A_i, \cdot) \to S_{n-1, p_0}(\A_0, \cdot)$ vaguely on $\Oc$.  
\end{lemma}
\begin{proof}  Let $g\in \mathbf{C}_c(\Oc)$ with compact support $\omega \subset \Oc$. We will follow the notations and the established facts in the proof of  Lemma \ref{vague-con-Sp}. In particular,  $h(\A_i \cap C_{t_\tau}, \cdot) \to h(\A_0 \cap C_{t_\tau}, \cdot)$ uniformly on $S^{n-1}$. For $i\in \mathbb{N}_0$, let
\begin{equation*}
\widetilde{\A}_i=[C, \overline{\omega}(\tau/2), -h(\A_i \cap C_{t_\tau}, \cdot)] \in \mathcal{K}(C, \overline{\omega}(\tau/2)),
\end{equation*} 
where $\tau=\inf\{\angle(u, v): u \in \omega \ \ \text{and} \ \ v  \in \partial \Oc\}>0$.
From Lemmas \ref{L56} and \ref{L2}, $\widetilde{\A}_i \to \widetilde{\A}_0$ as $ i \to \infty$  and $S_{n-1}(\widetilde{\A}_i, \cdot) \to S_{n-1}(\widetilde{\A}_0, \cdot)$ weakly on $\overline{\omega}(\tau/2)$. By \eqref{unif},  $h_C(\widetilde{\A}_i, \cdot) \to h_C(\widetilde{\A}_0, \cdot)$ uniformly on $\overline{\omega}(\tau/2)$. 
Note that $h_C(\widetilde{\A}_0, \cdot)$ is continuous and strictly negative on $\overline{\omega}(\tau/2)$. Thus, for $0\leq p_0\leq 1$, $[-h_C(\widetilde{\A}_0, \cdot)]^{1-p_0}$ is finite and bounded below by a positive number on $\overline{\omega}(\tau/2)$. Consequently,  $$g(u)[-h_C(\widetilde{\A}_i, \cdot)]^{1-p_i} \to g(u)[-h_C(\widetilde{\A}_0, \cdot)]^{1-p_0}$$ uniformly on $\overline{\omega}(\tau/2))$ as $i \to \infty$.  
It follows from Lemma \ref{L2} that \begin{equation}\label{weak-con-Spi}
\lim_{i\rightarrow \infty} \int_{\omega}  g(u)[-h_C(\widetilde{\A}_i, u)]^{1-p_i}  dS_{n-1} (\widetilde{\A}_i, u)=\int_{\omega}  g(u)[-h_C(\widetilde{\A}_0, u)]^{1-p_0}  dS_{n-1} (\widetilde{\A}_0, u). 
\end{equation}  
By \eqref{reverse-Gauss-image}, one gets, for $i\in \mathbb{N}_0$,  
$$\int_{\omega}  g(u)[-h_C(\A_i, u)]^{1-p_i}  dS_{n-1} (\A_i, u)=\int_{\omega}  g(u)[-h_C(\widetilde{\A}_i, u)]^{1-p_i}  dS_{n-1} (\widetilde{\A}_i, u). $$ Combining with \eqref{Sp} and \eqref{weak-con-Spi}, the following holds: 
\begin{align*} 
\lim_{i\rightarrow \infty} \int_{\omega} g(u) dS_{n-1, p_i} (\A_i, u)
&=\lim_{i\rightarrow \infty} \int_{\omega} g(u)[-h_C(\A_i, u)]^{1-p_i}  dS_{n-1} (\A_i, u)\\
& =\lim_{i\rightarrow \infty} \int_{\omega}  g(u)[-h_C(\widetilde{\A}_i, u)]^{1-p_i}  dS_{n-1} (\widetilde{\A}_i, u)\\ 
& =\int_{\omega}  g(u)[-h_C(\widetilde{\A}_0, u)]^{1-p_0}  dS_{n-1} (\widetilde{\A}_0, u)\\ 
& =\int_{\omega}  g(u)[-h_C(\A_0, u)]^{1-p_0} dS_{n-1} (\A_0, u)\\ 
&= \int_{\omega}  g(u) dS_{n-1, p_0} (\A_0, u). 
\end{align*}  
The desired vague convergence of $S_{n-1, p_i}(\A_i, \cdot) \to S_{n-1, p_0}(\A_0, \cdot)$   on $\Oc$ follows directly from the  arbitrariness of $g\in \mathbf{C}_c(\Oc)$. 
\end{proof}
  
As in Section \ref{sec-3},  
we let $\tau>0$ be such that $\overline{\omega}(\tau) \subset \Oc$ is compact and $\nu(\overline{\omega}(\tau))>0$. 
The following lemma can be established following along the same lines as the proof of Lemma \ref{widet0}. 

\begin{lemma}\label{hatt0}
Suppose that $p_i \rightarrow p_0$ as $i \rightarrow \infty$ and $\E_i$ satisfies equation \eqref{Ei-Sp}. Then, there exist two positive constants $\zeta_2$ and $\zeta_3$ such that for each $i\in \N_0$,
$$\zeta_2\leq b(\E_i) \leq \zeta_3.$$  
\end{lemma}

\begin{proof} 
We first prove the lower bound. By \eqref{Ei-Sp}, one can get $\nu(\overline{\omega}(\tau))=S_{n-1, p_i}(\E_i, \overline{\omega}(\tau))$ for $0\leq p_i\leq 1$.
Employing \eqref{Lp-surface-addition-1} to each $\E_i$, one has, for $0\leq p_i \leq 1$,   
\begin{align*} 
S(C_{t'}) \geq  b(\E_i)^{p_i-n}S_{n-1, p_i}(\E_i, \overline{\omega}(\tau)) = b(\E_i)^{p_i-n} \nu (\overline{\omega}(\tau)). 
\end{align*} 
A rearrangement of above inequality yields that, for all $i \in \N_0$,
\begin{equation*}
b(\E_i) \geq \left( \frac{\nu (\overline{\omega}(\tau))}{S(C_{t'})}\right)^{\frac{1}{n-p_i}}\geq \min\bigg\{ \left( \frac{\nu (\overline{\omega}(\tau))}{S(C_{t'})}\right)^{\frac{1}{n-1}},  \left( \frac{\nu (\overline{\omega}(\tau))}{S(C_{t'})}\right)^{\frac{1}{n}}\bigg\}.
\end{equation*} 
This gives the lower bound with $$\zeta_2=\min\bigg\{ \left( \frac{\nu (\overline{\omega}(\tau))}{S(C_{t'})}\right)^{\frac{1}{n-1}},  \left( \frac{\nu (\overline{\omega}(\tau))}{S(C_{t'})}\right)^{\frac{1}{n}}\bigg\}>0.$$
 
For the upper bound, employing  \eqref{unif-upp-01} to each $\E_i$, with the help of \eqref{Ei-Sp}, one has, for all $i \in \N_0$ and $0\leq p_i\leq 1$,   
\begin{align*} 
\nu(\Oc)\geq b(\E_i)^{n-p_i} \cdot \int_{\On}dv. 
\end{align*} 
This further gives, for all $i \in \N_0$, 
\begin{equation*}
b(\E_i)\leq \max\bigg\{1, \left(\frac{1}{\nu(\Oc)} \int_{\On}dv\right)^{\frac{1}{1-n}}\bigg\}. 
\end{equation*} 
The desired upper bound follows if we let $$\zeta_3= \max\bigg\{1, \left(\frac{1}{\nu(\Oc)} \int_{\On}dv\right)^{\frac{1}{1-n}}\bigg\}.$$ This completes the proof.  
\end{proof}

We are now in the position to prove our main theorem regarding the continuity of solutions to the $L_p$ Minkowski problem for $C$-close sets as $p_i\rightarrow p_0$. 

\begin{theorem}\label{con-p}  
Let $0\leq p_i\leq 1$ for $i \in \mathbb{N}_0$, and let $\nu$ be a nonzero finite Borel measure  defined on $\Omega_{C^\circ}$. Suppose that $\E_i$ for $i\in \mathbb{N}_0$ solves Problem \ref{Lp-Minkowski-unbounded} for $p_i$ (i.e., $\nu=S_{n-1, p_i}(\E_i, \cdot)$). Then $\E_i \rightarrow \E_0$  as $p_i\to p_0$. 
\end{theorem}
\begin{proof}
Let $0\leq p_i\leq 1$ and $i \in \N_0$. In order to show $\E_i \to \E_0$, we are required to show that every subsequence of $\{\E_i\}_{i \in \N}$ has a convergent subsequence with its limit being $\E_0$. Let $\{\E_{i_j}\}_{j \in \N} \subset \{\E_i\}_{i\in \N}$. By Lemma \ref{hatt0}, for all $j \in \N$, it follows that $\zeta_2 \leq b(\E_{i_j}) \leq \zeta_3$. Applying Theorem \ref{Blas-unb-1} to the sequence $\{\E_{i_j}\}_{j \in \N}$ implies that there exists a convergent subsequence, say $\{\E_{i_{j_k}}\}_{k \in \N} \subset \{\E_{i_j}\}_{j \in \N}$ such that $\E_{i_{j_k}} \to \widehat{\E}_0$ as $k \to \infty$. Together with Lemma \ref{vague-con-Spi}, one has  $S_{n-1, p_i}(\E_{i_{j_k}}, \cdot) \to S_{n-1, p_0}(\widehat{\E}_0, \cdot)$ vaguely on $\Oc$
as $k \to \infty$. Therefore $\nu=S_{n-1, p_0}(\widehat{\E}_0, \cdot)=S_{n-1, p_0}(\E_0, \cdot)$. The uniqueness in Theorem \ref{ex&uniC-close} gives $\widehat{\E}_0=\E_0$ and this completes the proof.
\end{proof}

\noindent{\bf{Acknowledgment.} }  The research of DY has been supported by an NSERC grant, Canada. The research of BZ has been supported by NSFC (No. 12371060), the Shaanxi Fundamental Science Research Project for Mathematics and Physics (No. 22JSZ012).

\noindent Wen Ai, \ \ {\tt wena@mun.ca} \\ 
\textit{Department of Mathematics and Statistics, Memorial
University of Newfoundland, St. John's, Newfoundland A1C 5S7, Canada.}
\vskip 2mm

\noindent Deping Ye, \ \ {\tt deping.ye@mun.ca}\\
\textit{Department of Mathematics and Statistics, Memorial
University of Newfoundland, St. John's, Newfoundland A1C 5S7, Canada.}
\vskip 2mm 

\noindent Baocheng Zhu, \ \ {\tt bczhu@snnu.edu.cn} \\ 
\textit{Department of Mathematics and Statistics, Shaanxi Normal University, Xi'an, 710119, China. }
\end{document}